\documentclass[11pt]{amsart}
\usepackage{amsfonts,amsthm,amsmath,amssymb}
\usepackage{amscd}
\usepackage[comma,sort,compress]{natbib}

\usepackage{hyperref} 
\hypersetup{ colorlinks=true, linkcolor=blue, citecolor=blue,
  filecolor=blue, urlcolor=blue}

\newtheorem{theorem}{Theorem}[section]
\newtheorem{lemma}{Lemma}[section]
\newtheorem{prop}{Proposition}[section]

\newtheorem{remark}{Remark}

\newcommand{\weak}{\stackrel{w}{\longrightarrow}}
\newcommand{\prob}{\stackrel{P}{\longrightarrow}}
\newcommand{\eid}{\stackrel{d}{=}}
\newcommand{\one}{{\bf 1}}

\newcommand{\reals}{{\mathbb R}}
\newcommand{\bbr}{\reals}
\newcommand{\vep}{\varepsilon}
\newcommand{\bbz}{\protect{\mathbb Z}}
\newcommand{\bbn}{\protect{\mathbb N}}
\newcommand{\bbc}{\protect{\mathbb C}}

\def\Var{{\rm Var}}

\DeclareMathOperator{\Tr}{Tr} 

\numberwithin{equation}{section}

\begin{document} 
\bibliographystyle{abbrvnat}

\title[Wigner matrices with dependent entries]{Limiting spectral
  distribution for Wigner matrices with dependent entries}

\author[A. Chakrabarty]{Arijit Chakrabarty}
\address{Theoretical Statistics and Mathematics Unit, Indian Statistical Institute, New Delhi}
\email{arijit@isid.ac.in}

\author[R. S. Hazra]{Rajat Subhra Hazra}
\address{Institut f\"ur Mathematik,Universit\"at  Z\"urich, Z\"urich}
\email{rajatmaths@gmail.com}

\author[D. Sarkar]{Deepayan Sarkar}
\address{Theoretical Statistics and Mathematics Unit, Indian Statistical Institute, New Delhi}
\email{deepayan@isid.ac.in}



\begin{abstract}
  In this article we show the existence of limiting spectral distribution of
  a symmetric random matrix whose entries come from a stationary Gaussian
  process with covariances satisfying a summability condition. We
  provide an explicit description of the moments of the limiting
  measure.  We also show that in some special cases the Gaussian assumption
  can be relaxed.  The description of the limiting measure can also be
  made via its Stieltjes transform which is characterized as the
  solution of a functional equation. In two special cases, we get a
  description of the limiting measure - one as a free product
  convolution of two distributions, and the other one as a dilation of
  the Wigner semicircular law.
\end{abstract}

\maketitle

\section{Introduction}
\label{sec:intro}


In his seminal paper, \citet{wig58} showed that for a symmetric random
matrix with independent on and off diagonal entries satisfying some
moment conditions, the empirical spectral distribution (henceforth
ESD) converges to the Wigner semicircle law (defined in
\eqref{example.eq2}, henceforth WSL).  Subsequent work has tried to
obtain a better understanding of the spectrum of such matrices, which
plays an important role in physics as well as other branches of
mathematics such as operator algebras.  Recently, there has been
interest in how far the independence assumption and the moment
conditions can be relaxed.  The reader may refer to the recent review
article by \citet{BA2011} and the references therein for an overview
of currently available results.

Relaxation of the independence assumption has been investigated by
\citet{ KM2008, RTWR, GT,cha}.  \cite{OE2012, adam}, and
\citet{HLN} have studied the sample covariance matrix imposing some
dependence on the rows and columns.  However, the limiting spectral
distributions (henceforth LSD) obtained by considering symmetric
matrices with the independence assumption weakened have stayed within
the WSL regime for the most part.  One exception is
\cite{AZ}, who considered the LSD of Wigner matrices where on and off
diagonal elements form a finite-range dependent random field; in
particular, the entries are assumed to be independent beyond a finite
range, and within the finite range the correlation structure is given
by a kernel function.

\subsection*{Motivation} 

We begin with a few examples to motivate the problem studied in this
article.  In each of the following examples, a random field
$\{Z_{i,j}:i,j\ge1\}$ is developed. For $n \ge 1$, let $A_n$ be the
$n\times n$ matrix whose $(i,j)$-th entry is $Z_{i\wedge j,i\vee
  j}$. The question is whether the ESD of $A_n/\sqrt n$ converges as
$n\to\infty$, and if so, where.

\noindent{\bf Example 1.} Let $\{Z_{i,j}:i,j\ge1\}$ be a mean zero
Gaussian process such that
$E\left[Z_{i,j}Z_{i+k,j+l}\right]=\rho^{|k|+|l|}$ for integers
$i,j,k,l$ such that $i,j,i+k,j+l \ge 1$,
where $|\rho|<1$ is fixed. This process can be thought of as a ``two
dimensional AR($1$) process'', because $\{Z_{i,j}:j\ge1\}$ is an
AR($1$) process for fixed $i$, as is $\{Z_{i,j}:i\ge1\}$ for fixed
$j$.

\noindent{\bf Example 2.} Assume that $\{G_{i,j}:i,j\ge1\}$ are
i.i.d. standard Gaussian random variables, and $N$ is a fixed positive
integer. Define
\[
 Z_{i,j}:=\sum_{k=0}^N\sum_{l=0}^NG_{i+k,j+l},\,i,j\ge1\,.
\]

\noindent{\bf Example 3.} Suppose that $(G_n:n\in\bbz)$ is a mean zero
variance one stationary Gaussian process. Let $(G_n^{i}:n\in\bbz)$ be
i.i.d. copies of $(G_n:n\in\bbz)$ for $i=\ldots,-2,-1,0,1,2,\ldots$.
Set $
Z_{i,j}:=G^{i-j}_{i},\,i,j\in\bbz\,.$


\noindent{\bf Example 4.} Let $\{c_{k,l}\}$ be real numbers such that
\begin{eqnarray*}
\sum_{k=-\infty}^\infty\sum_{l=-\infty}^\infty c_{k,l}^2&<&\infty\,,\\
c_{k,l}&=&c_{l,k}\mbox{ for all }k,l\in\bbz\,,\\
 \sum_{l=-\infty}^\infty c_{k,l}c_{k^\prime,l}&=&0\mbox{ for all }k\neq k^\prime\,.
\end{eqnarray*}
 As in Example 2, let  $\{G_{i,j}:i,j\ge1\}$ be i.i.d. standard Gaussian random variables. Define
\[
 Z_{i,j}:=\sum_{k,l\in\bbz}c_{k,l}G_{i-k,j-l},\,\,i,j\in\bbz\,.
\]

It is shown later in Section \ref{sec:example} that for Examples 1 and 2, the LSD of $A_n/\sqrt n$
is the free product convolution of the WSL with a distribution
supported on a compact subset of $[0,\infty)$, and for Examples 3
(under the additional assumption that
$\sum_{n=1}^\infty|E(G_0G_n)|<\infty$) and 4, the LSD is a dilation of
the WSL.  To the best of our understanding, Example 2 is the only one
of the above examples where the result follows from the work of
\cite{AZ}, because that is the only example where two entries are
independent if their distance is above a threshold.

Motivated by these examples, this article considers a random matrix
model where on and off diagonal entries form a stationary Gaussian
field, with the covariance of the entries satisfying a summability
condition.  
It is shown using the method of moments that the ESD
converges to a non-degenerate measure.  The combinatorial approach we
have adopted for calculating the traces of powers of the matrices
avoids the use of independence in any stage.
Unlike \cite{AZ} where the use of independence
facilitates the negligibility of certain partitions, sharper estimates
are needed on a class of partitions.  These sharper estimates on the
set of partitions and Wick's formula are used to derive the limiting
moments. The assumption of Gaussianity, although important in the proof, is relaxed
to allow for a fairly general class of input sequences using the
Lindeberg type argument developed in \citet{chatterjee:2005}.  An
interpretation of the limiting moments in terms of functions of
non-crossing pair partitions is used to derive the Stieltjes transform
of the measure.  The form of the Stieltjes transform indicates a
relationship with operator-valued semicircular variables studied in
\citet{speiO} (for the application of free probability to random
matrices, see the recent review by \citet{spei-review}).  



\subsection*{Outline of our contribution}

Let $(Z_{i,j}:i,j\in\bbz)$ be a stationary, mean zero, variance one
Gaussian process. Stationarity here means that for $k,l\in\bbz$,
\[
 (Z_{i+k,j+l}:i,j\in\bbz)\eid (Z_{i,j}:i,j\in\bbz)\,.
\]
For $i,j\ge1$, set
\[
 X_{i,j}:=Z_{i\wedge j,i\vee j}\,,
\]
and let 
\begin{equation}\label{gaussian.defA_n}
 A_n:=((X_{i,j}))_{n\times n},\,n\ge1\,.
\end{equation}
Let $\lambda_1\le\ldots\le\lambda_n$ be the eigenvalues of $A_n$,
which are real because $A_n$ is symmetric, and denote
\begin{equation}\label{gaussian.defmu_n}
 \mu_n:=\frac1n\sum_{i=1}^n\delta_{\{\lambda_i/\sqrt n\}}\,.
\end{equation}
The main result of this article is Theorem~\ref{gaussian.t1}, stated
in Section \ref{sec:mainresult} along with an outline of the proof,
which gives a set of conditions on the covariance of $\{X_{i,j}\}$
under which the ESD $\mu_n$ converges weakly in probability. In
Section \ref{sec:combinatorics}, some combinatorial results are
proven, which are used in Section \ref{sec:gaussian} for the proof of
Theorem~\ref{gaussian.t1}. In Section \ref{sec:inv}, we show that by
specializing on an infinite order moving average process with
independent inputs satisfying the Pastur condition,
Theorem~\ref{gaussian.t1} and an invariance principle can be used to
establish the convergence of the ESD.  In Section \ref{sec:stieltjes},
an explicit description of the Stieltjes transform is provided using
the moment formula and some properties of the Kreweras complement.
In Section \ref{sec:example}, two explicit examples are
described where we get better descriptions of the limit: Theorem
\ref{example.t1} gives conditions under which the LSD is the free
multiplicative convolution of the WSL and another distribution.
Theorem \ref{example.t2} gives conditions under which the LSD is the
WSL. Finally, in Section \ref{sec:ext}, Theorem \ref{example.t1} is extended to entail the case where the correlations are not necessarily summable. That assumption is replaced there by the weaker assumption of absolute continuity of the spectral measure.

\section{The main result}\label{sec:mainresult} 

In this section, we state the main result, and give an outline of the
proof. Let the $n\times n$ random symmetric matrix $A_n$ be as in
\eqref{gaussian.defA_n}, and set $\mu_n$ to be ESD of $A_n/\sqrt n$, as defined in \eqref{gaussian.defmu_n}.
Before stating the main result, we need a few more notations and
assumptions.  Define
\begin{equation}\label{gaussian.defR}
 R(u,v):=E\left[Z_{0,0}Z_{-u,v}\right],\,u,v\in\bbz\,.
\end{equation}
%

The assumptions are the following.\\
{\bf Assumption 1:} $R(\cdot,\cdot)$ is symmetric, that is,
\begin{equation}\label{eq.symmetry}
 R(u,v)=R(v,u)\mbox{ for all }u,v\in\bbz\,.
\end{equation}
{\bf Assumption 2:} $R(\cdot,\cdot)$ is absolutely summable, that is,
\begin{equation}\label{eq.defrbar}
 \bar R:=\sum_{u,v\in\bbz}|R(u,v)|<\infty\,.
\end{equation}

An immediate consequence of Assumption 1 and stationarity is that
\begin{equation}\label{gaussian.eq13}
 R(u,v)=R(-v,-u),\,u,v\in\bbz\,.
\end{equation}

A consequence of Assumption 2 is the following.  Fix $\sigma\in
NC_2(2m)$, the set of non-crossing pair partitions of
$\{1,\ldots,2m\}$. Let $(V_1,\ldots,V_{m+1})$ denote the Kreweras
complement of $\sigma$, which is the maximal partition
$\overline\sigma$ of $\{\overline 1,\ldots,\overline {2m}\}$ such that
$\sigma\cup\overline\sigma$ is a non-crossing partition of
$\{1,\overline1,\ldots,2m,\overline{2m}\}$. For $1\le i\le m+1$,
denote
\begin{equation}\label{eq.defV}
 V_i:=\{v_1^i,\ldots,v_{l_i}^i\}\,.
\end{equation}
Define
\begin{equation}\label{eq.defS}
 S(\sigma):=\left\{(k_1,\ldots,k_{2m})\in\bbz^{2m}:\sum_{j=1}^{l_s}k_{v^s_j}=0,\,s=1,\ldots,m+1\right\}\,.
\end{equation}
If $\sigma=\{(u_1,u_{m+1}),\ldots,(u_m,u_{2m})\}$, then notice that
\begin{eqnarray}
\nonumber&& \sum_{(k_1,\ldots,k_{2m})\in S(\sigma)}\prod_{(u,v)\in\sigma}|R(k_u,k_v)|
\end{eqnarray}
\begin{eqnarray}
\nonumber&=&\sum_{i\in\bbz^{2m}}\Biggl[\#\Bigl\{k\in S(\sigma):\mbox{ there exists a permutation }\pi\mbox{ of }\{1,\ldots,m\}\\
\nonumber&&\,\,\,\mbox{ such that }k_{u_{j}}=i_{\pi(j)}\mbox{ and }k_{u_{m+j}}=i_{m+\pi(j)},\,j=1,\ldots,m\Bigr\}\\
\nonumber&&\,\,\,\,\,\,\,\,\,\,\,\,\prod_{j=1}^m|R(i_j,i_{m+j})|\Biggr]\\
\nonumber&\le&\sum_{i\in\bbz^{2m}}\Biggl[m!\prod_{j=1}^m|R(i_j,i_{m+j})|\Biggr]\\
\label{eq.ub}&=&m!{\bar R}^m<\infty\,.
\end{eqnarray}
In view of the above calculation, it makes sense to define
\begin{equation}\label{eq.defbeta}
 \beta_{2m}:=\sum_{\sigma\in NC_2(2m)}\sum_{k\in S(\sigma)}\prod_{(u,v)\in\sigma}R(k_u,k_v),\,m\ge1\,.
\end{equation}

The main result of this article is the following.

\begin{theorem}\label{gaussian.t1}
  Under Assumption 1 and Assumption 2, $\mu_n$ converges weakly in
  probability to a distribution $\mu$. The $k$-th moment of $\mu$ is
  zero if $k$ is odd, and $\beta_{k}$ if $k$ is even. Furthermore,
  $\mu$ is uniquely determined by its moments, that is, if a
  distribution has the same moments as that of $\mu$, then the
  distribution equals $\mu$.
\end{theorem}
\begin{remark}
 As is common in the literature, the phrase ``$\mu_n$ converges weakly in probability to a distribution $\mu$'' means that
\[
 L\left(\mu_n,\mu\right)\prob0\,,
\]
as $n\to\infty$, where $L$, the L\'evy distance, is defined by
\begin{equation}\label{eq.deflevy}
 L(\nu_1,\nu_2):=\inf\Bigl\{\vep>0:\nu_1\left((-\infty,x-\vep]\right)-\vep\le\nu_2\left((-\infty,x]\right)\le
\end{equation}
\[
 \nu_1\left((-\infty,x+\vep]\right)+\vep\mbox{ for all }x\in\bbr\Bigr\}\,,
\]
for probability measures $\nu_1,\nu_2$ on $\bbr$.
\end{remark}

We end this section with a brief outline of the proof of the above result. As is standard in a proof by the method of moments, what needs to be shown is that for fixed $m\ge1$,
\begin{equation}\label{main.eq1}
 \lim_{n\to\infty}n^{-(m+1)}\sum_{i_1,\ldots,i_{2m}=1}^nE\left[X_{i_1,i_2}\ldots X_{i_{2m-1},i_{2m}}X_{i_{2m},i_1}\right]=\beta_{2m}\,.
\end{equation}
As in the proof of the classical Wigner's result, the first step is to get rid of the ``non-pair matched'' tuples $i=(i_1,\ldots i_{2m})$ in the above sum. Fix $N\ge1$, and say that a tuple $i\in\{1,\ldots,n\}^{2m}$ is $N$-pair matched if there exists a pair partition $\pi$ of $\{1,\ldots,2m\}$ such that for all $(u,v)\in\pi$,
\[
 \left|i_{u-1}\wedge i_u-i_{v-1}\wedge i_v\right|\vee\left|i_{u-1}\vee i_u-i_{v-1}\vee i_v\right|\le N\,,
\]
with the convention $i_0:=i_{2m}$. It needs to be shown that if $C_{N,n}$ denotes the set of tuples in $\{1,\ldots,n\}^{2m}$ which are not $N$-pair matched, then
\begin{equation}\label{main.eq2}
 \lim_{N\to\infty}\limsup_{n\to\infty}n^{-(m+1)}\left|\sum_{i\in C_{N,n}}^nE\left[X_{i_1,i_2}\ldots X_{i_{2m-1},i_{2m}}X_{i_{2m},i_1}\right]\right|=0\,.
\end{equation}
Unlike in the classical Wigner's result, this is a non-trivial step in
our situation because not only does the above sum not vanish for $N$
large, even showing that the expectation in modulus is less than some
$\vep$ is not enough because $\#C_{N,n}\sim n^{2m}$ as $n\to\infty$,
and the sum is scaled only by $n^{(m+1)}$. This is precisely the step
where Assumption 2 plays an important role.

Once \eqref{main.eq2} is established, what remains to be shown for \eqref{main.eq1} is that for fixed $N$,
\begin{equation}\label{main.eq3}
 \lim_{n\to\infty}n^{-(m+1)}\sum_{i\in C_{N,n}^c}^nE\left[X_{i_1,i_2}\ldots X_{i_{2m-1},i_{2m}}X_{i_{2m},i_1}\right]
\end{equation}
\[
 =\sum_{\sigma\in NC_2(2m)}\,\sum_{k\in S(\sigma):\max_j|k_j|\le N}\,\prod_{(u,v)\in\sigma}R(k_u,k_v)\,.
\]
By standard combinatorial arguments, the sum over $C_{N,n}^c$ can be shown to be asymptotically equivalent to the sum over all tuples that are Catalan with respect to some $\sigma\in NC_2(2m)$, that is,  whenever $(j,k)\in\sigma$,
\begin{equation*}
 |i_{j-1}-i_{k}|\vee|i_j-i_{k-1}|\le N\,.
\end{equation*}
The final step is to show that for fixed $\sigma\in NC_2(2m)$, if $D_\sigma$ denotes the set of tuples in $\{1,\ldots,n\}^{2m}$ which are Catalan with respect to $\sigma$, then
\[
 \lim_{m\to\infty}\sum_{i\in D_\sigma}E\left[X_{i_1,i_2}\ldots X_{i_{2m-1},i_{2m}}X_{i_{2m},i_1}\right]
\]
\[
 =\sum_{k\in S(\sigma):\max_j|k_j|\le N}\,\prod_{(u,v)\in\sigma}R(k_u,k_v)\,.
\]
This follows by computing the expectation via Wick's formula, and observing that in that formula, the contribution of all the pair partitions excluding $\sigma$ is asymptotically negligible.
This final step establishes \eqref{main.eq3}.

\section{Some combinatorics}\label{sec:combinatorics}

In this section, we recall some elementary combinatorial notions, and
prove a few results related to them. The results of this section are
not of independent interest, but will be used in Section
\ref{sec:gaussian}.

There is an infinite totally ordered set called the ``alphabet'' whose
elements are called ``letters''. The order with which the alphabet is
endowed is the ``alphabetical ordering''. A ``word'' is an ordered
finite collection of not necessarily distinct letters. While the
actual description of the alphabet is irrelevant, to fix ideas, we
shall consider the set of natural numbers with the natural ordering to
be the alphabet. Two words are ``distinct'' if one cannot be obtained
from the other by relabeling letters. For example, the words 1112 and
2224 are not distinct, and the words 4612 and 2181 are distinct. If
the lengths of two words are different, then they are necessarily
distinct.

Let $P(2m)$ and $NC_2(2m)$ denote the set of pair partitions and
non-crossing pair partitions of $\{1,\ldots,2m\}$
respectively. Clearly, $NC_2(2m)$ is a proper subset of $P(2m)$ for
$m\ge2$. For example, $\{(1,3),(2,4),(5,6)\}\in P(6)\setminus NC_2(6)$
and $\{(1,4),(2,3),(5,6)\}\in NC_2(6)$. Recall that for $\sigma\in
NC_2(2m)$, the Kreweras complement $K(\sigma)$ is the maximal
partition $\overline\sigma$ of $\{\overline 1,\ldots,\overline {2m}\}$
such that $\sigma\cup\overline\sigma$ is a non-crossing partition of
$\{1,\overline1,\ldots,2m,\overline{2m}\}$. For example,
\[
K\left(\{(1,4),(2,3),(5,6)\}\right)=\{(\overline 1,\overline
3),(\overline 2),(\overline 4,\overline 6),(\overline 5)\}\,.
\]

For $m\ge1$, $\pi$ is a ``pairing'' of $\{1,\ldots,2m\}$ if it is a
permutation of that set satisfying
\[
 \pi(j)\neq j=\pi(\pi(j)),\,1\le j\le2m\,.
\]
Call $\pi$ to be an ``almost pairing'' of $\{1,\ldots,2m\}$ if it is a
permutation satisfying
\[
 \#\{1\le j\le2m:\pi(j)=j\}=2\,,
\]
and
\[
 \pi(\pi(j))=j,\,1\le j\le2m\,.
\]
There is a clear bijection between the set of pairings and $P(2m)$,
namely for any pairing $\pi$, $\{(j,\pi(j)):1\le j\le2m\}\in
P(2m)$. Keeping this bijection in mind, we shall use the words pairing
and pair partition interchangeably.

Recall that a word $A:=a_1\ldots a_{2m}$ is ``pair matched'' if there
exists a pairing $\pi$ of $\{1,\ldots,2m\}$ such that
\[
 a_j=a_{\pi(j)},\,1\le j\le2m\,.
\]
If there exists an almost pairing $\pi$ satisfying the above, then $A$
is ``almost pair matched''. Examples of pair matched words are 1212,
122221 etc., and that of almost pair matched words are 1231, 2111,
1221 etc. An example of a non-pair matched word is 111222. Notice that
while the set of pairings and almost pairings are disjoint, a pair
matched word is necessarily almost pair matched. Another useful
observation is that a word is almost pair matched if and only if it be
can made pair matched by changing at most one letter.

Recall that a pair matched word is a ``Catalan word'' if successive
deletions of double letters lead to the empty word. Examples of
Catalan words are 1221, 123321, 122122 etc., while 1212 is an example
of a pair-matched word which is not a Catalan word.

The conventions that we now discuss will be assumed throughout the
article. Any tuple $i\in\bbz^k$ is taken to be of the form
$i:=(i_1,\ldots,i_k),$ and furthermore, implicitly defines
$i_0:=i_k$.  The same convention also applies to words, that is, for
a word $A=a_1\ldots a_k$,
$a_0:=a_k$.
The next convention is that, for alphabets or integers $a,b,c,d$, we say
\begin{equation}\label{eq.conv2}
 (a,b)\approx(c,d)
\end{equation}
if $a\wedge b = c\wedge d$ and $a\vee b = c\vee d$.  For integers
``$\wedge$'' and ``$\vee$'' have the usual interpretation of minimum
and maximum respectively, whereas for letters $u$ and $v$, $u\wedge v$
and $u\vee v$ mean the one that comes first in the alphabetical
ordering, and the one that comes later, respectively.

Given words $A:=a_1\ldots a_m$ and $B:=b_1\ldots b_m$ { of the same
  length}, say that $B$ is an ``offspring'' of $A$ if the following is
true. Whenever $a_j=a_k$ for some $1\le j,k\le m$, it holds that
\[
\left(b_{j-1},b_j\right)\approx\left(b_{k-1},b_k\right)\,.
\]
For example, 1213 is an offspring of both 1221 and 5789, and both 7565
and 1111 are offsprings of 2211.

The next notion we need is that of a ``compound offspring
word''.  A word $B=b_1\ldots b_{2m}$ is a compound offspring of
$A=a_1\ldots a_{2m}$ if $b_1\ldots b_m$ and $b_{m+1}\ldots b_{2m}$ are
offsprings of $a_1\ldots a_m$ and $a_{m+1}\ldots a_{2m}$ respectively,
and furthermore, whenever $a_j=a_k$, it holds that
\[
 \left(b_{j}, b_{\gamma(j)}\right)\approx\left(b_{k}, b_{\gamma(k)}\right)\,,
\]
where
\[
 \gamma(j):=\left\{
\begin{array}{ll}
 j-1,&j\in\{1,\ldots,2m\}\setminus\{1,m+1\}\,,\\
m,&j=1\,,\\
2m,&j=m+1\,.
\end{array}
\right.
\]

For a word $A$, we denote by $\#A$ the number of distinct letters in $A$ (and {\bf not the length} of $A$). 

The following result is well known in the literature, but in different settings. One can look at, for example, equation (34) in the proof of Theorem 4 in \cite{bose:sen:2008} where the same claim has been restated in a slightly different language. Hence we omit the proof.

\begin{lemma}\label{comb.l0}
(a) Let $A$ be a pair matched word of length $2m$ for some $m\ge1$. Then $A$ has an offspring word $B$ with
\begin{equation}\label{comb.l0.eq1}
 \#B=m+1
\end{equation}
if and only if $A$ is a Catalan word with
\[
 \#A=m\,.
\]
In this case, the offspring word is unique upto relabeling of letters.\\
(b) Assume that $A_1$ and $A_2$ are distinct Catalan words of length $2m$, that is, one cannot be obtained from the other by relabeling letters. If $B_1$ and $B_2$ are offsprings of $A_1$ and $A_2$ respectively such that
\[
 \#B_1=\#B_2=m+1\,,
\]
then $B_1$ and $B_2$ are distinct.\\
(c) Furthermore, if $B=b_1\ldots b_{2m}$ is an offspring of $A=a_1\ldots a_{2m}$ satisfying \eqref{comb.l0.eq1}, then it is necessary that whenever $a_j=a_k$ for some $j<k$, it holds that
\begin{eqnarray}
 b_{j-1}&=&b_k\,,\mbox{ and }\label{comb.l0.eq2}\\
b_j&=&b_{k-1}\,.\label{comb.l0.eq3}
\end{eqnarray}
\end{lemma}

As mentioned at the beginning of the section, the remaining results will be used for later results in this section or the ones in Section \ref{sec:gaussian}. 

\begin{lemma}\label{gaussian.l0}
Let $B=b_1\ldots b_{2k}$ be a compound offspring word of $A=a_1\ldots a_{2k}$ for some $k\ge2$. Define
\[
 n_j:=\left\{\begin{array}{ll}
	      \#(b_kb_1\ldots b_{j})-\#(a_1\ldots a_{j}),&1\le j\le k-1\,,\\
	      \#(b_{2k}b_1\ldots b_{j})-\#(a_1\ldots a_{j}),&k+1\le j\le 2k-1\,.
             \end{array}
\right.
\]
Then,
\begin{equation}\label{gaussian.l0.mainclaim}
1\ge n_1\ge\ldots\ge n_{k-1}\ge n_{k+1}-1\ge\ldots\ge n_{2k-1}-1\,.
\end{equation}
Consequently, if $D$ is an offspring word of $C:=c_1\ldots c_k$ for some $k\ge2$, then
\begin{equation}\label{gaussian.l0.claim}
 \#D\le1+\#(c_1\ldots c_{k-1})\,.
\end{equation}
\end{lemma}
\begin{proof}
 It is easy to see that the inequality $n_{k-1}\le1$ in \eqref{gaussian.l0.mainclaim} implies \eqref{gaussian.l0.claim}. So the former claim is the one that needs a proof. The leftmost inequality in \eqref{gaussian.l0.mainclaim} is trivial. For the subsequent inequalities, fix $1\le j\le k-2$, and we shall show that
\begin{equation}\label{gaussian.l0.eq1}
 n_{j+1}\le n_j\,.
\end{equation}
The proof will be separate for the two cases:
\[
 \mbox{\bf(Case 1) } a_{j+1}\neq a_i\mbox{ for all }1\le i\le j\,,
\]
and
\[
 \mbox{\bf(Case 2) } a_{j+1}=a_i\mbox{ for some }1\le i\le j\,.
\]
Observe that in Case 1,
\begin{eqnarray*}
 n_{j+1}&=&\#(b_{k}b_1\ldots b_{j+1})-\#(a_1\ldots a_{j+1})\\
&\le&1+\#(b_{k}b_1\ldots b_{j})-\#(a_1\ldots a_{j+1})\\
&=&\#(b_{k}b_1\ldots b_{j})-\#(a_1\ldots a_{j})\\
&=&n_j\,.
\end{eqnarray*}
In Case 2, if $i\ge2$, then $b_{j+1}$ equals $b_i$ or $b_{i-1}$, and if $i=1$, then $b_{j+1}$ equals $b_i$ or $b_k$.
Hence
\begin{eqnarray*}
 \#(b_kb_1\ldots b_{j+1})&=&\#(b_kb_1\ldots b_{j})\,,\\
\mbox{and }\#(a_1\ldots a_{j+1})&=&\#(a_1\ldots a_{j})\,,
\end{eqnarray*}
which shows that
\[
 n_{j+1}=n_j\,.
\]
This establishes \eqref{gaussian.l0.eq1} for $1\le j\le k-2$. Similar arguments establish the same claim for $k+1\le j\le2k-2$, and that
\[
 n_{k+1}\le1+n_{k-1}\,.
\]
This completes the proof of \eqref{gaussian.l0.mainclaim}, and thereby establishes the lemma.
\end{proof}

\begin{lemma}\label{gaussian.l1}
 Suppose that $A$ is an almost pair matched word of length $2m$ with $m\ge2$, and $B$ is an offspring of $A$. Then, 
\[
 \#B\le m+1\,.
\]
\end{lemma}

\begin{proof}Since $A$ is almost pair matched, clearly $\#A\le m+1$.
If $\#A\le m$, then by \eqref{gaussian.l0.claim}, it follows that
\[
 \#B\le1+\#A\le1+m\,.
\]

So, without loss of generality, let us assume that $\#A=m+1$.  We
start with the observation that if $b_1\ldots b_{2m}$ is an offspring
of $a_1\ldots a_{2m}$, then $b_2\ldots b_{2m}b_1$ is an offspring of
$a_2\ldots a_{2m}a_1$. Therefore, once again without loss of
generality, we can and do assume that
\[
 A=W_1cW_2d\,,
\]
where $W_1$ and $W_2$ are possibly empty words such that $W_1W_2$ is pair matched, and $c,d$ are distinct letters which do not occur in $W_1W_2$. Therefore, by \eqref{gaussian.l0.claim},
\begin{eqnarray*}
 \#B&\le&1+\#(W_1cW_2)\\
&=&1+m\,.
\end{eqnarray*}
This completes the proof.
\end{proof}

\begin{lemma}\label{comb.l2}
 Suppose that $A$ is an almost pair matched word of length $2m$ with $m\ge2$, and $B$ is a compound offspring of $A$. Then, 
\[
 \#B\le m+2\,.
\]
\end{lemma}

\begin{proof}
  Denote $A = a_1\ldots a_{2m}$.  By the fact that the rightmost
  quantity in \eqref{gaussian.l0.mainclaim} is at most $1$, it follows
  that
  \begin{equation}\label{gaussian.l2.toshow}
    \#B\le2+\#(a_1\ldots a_{2m-1})\,.
  \end{equation}
  Thus, the claim of the lemma follows if $\#A\le m$. So assume
  without loss of generality that
  \[
  \#A=m+1\,.
  \]
  It is easy to see that since $b_1\ldots b_{2m}$ is a compound
  offspring of $a_1\ldots a_{2m}$, so are $b_{m+1}\ldots
  b_{2m}b_1\ldots b_m$ and $b_1\ldots b_mb_{m+2}\ldots b_{2m}b_{m+1}$
  of\\ $a_{m+1}\ldots a_{2m}a_1\ldots a_m$ and $a_1\ldots
  a_ma_{m+2}\ldots a_{2m}a_{m+1}$ respectively. Therefore, as in the
  proof of Lemma \ref{gaussian.l1}, we assume without loss of
  generality that
\begin{equation}\label{comb.l2.wlog}
 a_{2m}\neq a_j,\,j=1,\ldots,2m-1\,.
\end{equation}
Clearly, \eqref{gaussian.l2.toshow} and the observation that under
this assumption 
\[
\#(a_1\ldots a_{2m-1})=m\,, 
\]
 completes the proof.
\end{proof}

\begin{lemma}\label{comb.l1}
 Suppose that $A$ is a pair matched word of length $4m$ for some $m\ge1$. Then, $A$ has a compound offspring word $B$ with
\begin{equation}\label{comb.l1.hypo}
 \#B=2m+2
\end{equation}
if and only if 
\begin{equation}\label{comb.l1.toshow}
 A=A_1A_2
\end{equation}
where $A_1$ and $A_2$ are Catalan words of length $2m$ with no common letters, and
\begin{equation}\label{comb.l1.toshow2}
 \#A_1=\#A_2=m\,.
\end{equation}
In this case,
\[
 B=B_1B_2\,,
\]
where $B_1$ and $B_2$ are offspring words of $A_1$ and $A_2$ respectively, do not have a common letter, and satisfy
\begin{equation}\label{comb.l1.toshow3}
 \#B_1=\#B_2=m+1\,.
\end{equation}
Furthermore, the compound offspring word $B$ satisfying \eqref{comb.l1.hypo} is unique up to relabeling, and if $B$ and $B^\prime$ are compound offspring words of distinct pair matched words of length $4m$ satisfying
\[
 \#B=\#B^\prime=2m+2\,,
\]
then $B$ and $B^\prime$ are distinct.
\end{lemma}

\begin{proof}
 Assume for a moment that the ``if and only if'' claim has been shown. Let $B$ be a compound offspring of $A$ such that \eqref{comb.l1.hypo} holds. By definition of a compound offspring word, we can write
\[
 B=B_1B_2
\]
where $B_1$ and $B_2$ are offspring words of $A_1$ and $A_2$ respectively. Now notice that by \eqref{comb.l1.toshow2} and Lemma \ref{comb.l2},
\[
 \#B_i\le m+1,\,i=1,2\,.
\]
Thus,
\begin{eqnarray*}
 2m+2&=&\#B\\
&\le&\#B_1+\#B_2\\
&\le&2m+2\,.
\end{eqnarray*}
Therefore, $B_1$ and $B_2$ cannot have a common letter, because otherwise, the inequality in the second line becomes strict. Also, the inequality in the last line must be an equality, proving \eqref{comb.l1.toshow3}. The final claim follows from Lemma \ref{comb.l0} (b).

So the ``if and only if'' claim is the only part that needs a proof. Once again, the ``if'' part follows trivially from  Lemma \ref{comb.l0} (a). Let us proceed towards the ``only if'' part. So assume that $A$ has a compound offspring word $B$ such that \eqref{comb.l1.hypo} holds. Let $A:=a_1\ldots a_{4m}$ and $B=b_1\ldots b_{4m}$. Set
\begin{eqnarray*}
 A_1&:=&a_1\ldots a_{2m}\,,\\
A_2&:=&a_{2m+1}\ldots a_{4m}\,.
\end{eqnarray*}
Thus, \eqref{comb.l1.toshow} trivially holds. We start with showing that $A_1$ and $A_2$ do not have a common letter. Assume for the sake of contradiction that they have a common letter. The arguments that justify the assumption \eqref{comb.l2.wlog} in the proof of Lemma \ref{comb.l2} show that in this case, $a_{2m+1}$ can be chosen to be that letter without loss of generality, that is,
\begin{equation}\label{comb.l1.eq4}
 a_{2m+1}=a_j\mbox{ for some }1\le j\le2m\,.
\end{equation}
Therefore, 
\begin{equation}\label{comb.l1.eq1}
 (b_{4m},b_{2m+1})\approx(b_j,b_{\gamma(j)})\,,
\end{equation}
where ``$\approx$'' and $\gamma(\cdot)$ are as in \eqref{eq.conv2} and the definition of a compound offspring word respectively.
Let $n_j$ for $j=1,\ldots,2m-1,2m+1,\ldots,4m-1$, be as in the statement of Lemma \ref{gaussian.l0} with $k=2m$.
By that result, it follows that
\[
 1\ge n_1\ge\ldots\ge n_{2m-1}\,,
\]
and
\[
 n_{2m+1}\ge\ldots\ge n_{4m-1}\,.
\]
Thus,
\begin{eqnarray*}
n_{4m-1}&\le& n_{2m+1}\\
&=&\#(b_{4m}b_1\ldots b_{2m+1})-\#(a_1\ldots a_{2m+1})\\
&=&\#(b_1\ldots b_{2m})-\#(a_1\ldots a_{2m+1})\\
&\le&\#(b_1\ldots b_{2m})-\#(a_1\ldots a_{2m-1})\\
&=&n_{2m-1}\\
&\le&1\,,
\end{eqnarray*}
the equality in the third line following by \eqref{comb.l1.eq1}. However, notice that
\begin{equation}\label{comb.l1.eq3}
 n_{4m-1}=\#B-\#(a_1\ldots a_{4m-1})=\#B-\#A\ge2\,,
\end{equation}
the second equality following from the fact that $A$ is pair matched, and the inequality following from \eqref{comb.l1.hypo}. This clearly is a contradiction, thus showing that $A_1$ and $A_2$ have no common letters. An immediate consequence is that $A_1$ and $A_2$ are pair matched words. 

Next, we proceed towards showing \eqref{comb.l1.toshow2}.
Lemma \ref{gaussian.l0} implies that
\begin{equation}\label{comb.l1.eq2}
 1\ge n_{2m-1}\ge n_{2m+1}-1\ge n_{4m-1}-1\ge1\,,
\end{equation}
the rightmost one following from \eqref{comb.l1.eq3} which is clearly valid regardless of the assumption \eqref{comb.l1.eq4}. Thus,
\[
\#A= \#B-n_{4m-1}=2m\,.
\]
Hence each letter in $A$ comes exactly twice, and so \eqref{comb.l1.toshow2} holds. 
Another consequence of \eqref{comb.l1.eq2} is that
\[
 n_{2m-1}=1\,,
\]
a restatement of which in view of the fact that $A_1$ is pair matched is
\[
 \#(b_1\ldots b_{2m})=1+\#A_1=m+1\,.
\]
Since $b_1\ldots b_{2m}$ is an offspring word of $A_1$, by Lemma \ref{comb.l0}, it follows that $A_1$ is a Catalan word. A similar argument holds for $A_2$, and completes the proof.
\end{proof}

\begin{lemma}\label{gaussian.l2}
 Let $A:=A_1\ldots A_{2m}$ be an almost pair matched word of length $2m$ where $m\ge2$. Assume that $A$ has $m+1$ distinct letters $a_1,\ldots,a_{m+1}$ with each of $a_1,\ldots,a_{m-1}$ occurring twice. Fix
\[
 u:=(u_1,\ldots,u_{m-1}),v:=(v_1,\ldots,v_{m-1})\in\bbz^{m-1}\,.
\]
Given a $2m$-tuple $(i_1,\ldots,i_{2m})$ in $\bbn^{2m}$, say that it is $uv$-matched if the following is true:
\[
 \mbox{whenever }A_j=A_k=a_l\mbox{ for some }1\le j<k\le2m\mbox{ and }1\le l\le m-1\,,
\]
it holds that
\[
 (i_{j-1}\wedge i_j)-(i_{k-1}\wedge i_k)=u_l\,,
\]
and
\[
 (i_{j-1}\vee i_j)-(i_{k-1}\vee i_k)=v_l\,,
\]
where $i_0:=i_{2m}$, as usual. Let $U_n$ denote the set of $uv$-matched tuples in $\{1,\ldots,n\}^{2m}$ for $n\ge1$. Then, 
\[
 \#U_n\le4^mn^{m+1}\mbox{ for all }n\ge1\,.
\]
\end{lemma}

\begin{proof}
 Let $\Pi$ denote the set of all functions from $\{1,2,\ldots,2m\}$ to $\{0,1\}$. Clearly, if $(i_1,\ldots,i_{2m})\in U_n$, then the following is true. There exists $\pi\in\Pi$ such that whenever $A_j=A_k=a_l$ for some $1\le j<k\le2m$ and $1\le l\le m-1$, it holds that
\begin{equation}\label{gaussian.eq3}
 i_{j-\pi(j)}-i_{k-\pi(k)}=u_l\,,
\end{equation}
and
\begin{equation}\label{gaussian.eq4}
 i_{j-(1-\pi(j))}-i_{k-(1-\pi(k))}=v_l\,.
\end{equation}
It is easy to see that the number of $(i_1,\ldots,i_{2m})$ in $\{1,\ldots,n\}^{2m}$ satisfying \eqref{gaussian.eq3} and \eqref{gaussian.eq4} is at most the number of those satisfying the same equations with $u_l$ and $v_l$ replaced by $0$ for all $l$.

Fix $\pi\in\Pi$. We shall now show that the number of $i:=(i_1,\ldots,i_{2m})\in\{1,\ldots,n\}^{2m}$ satisfying \eqref{gaussian.eq3} and \eqref{gaussian.eq4} with $u_l$ and $v_l$ replaced by $0$ for all $l$ is at most $n^{m+1}$. Clearly, for any such $i$, the word $B=i_1i_2\ldots i_{2m}$ is an offspring word of $A$. Since $B$ can have at most $m+1$ distinct letters by Lemma \ref{gaussian.l1}, it follows that the number of such $i$'s is at most $n^{m+1}$. Thus,
\[
 \#U_n\le n^{m+1}\#\Pi=4^mn^{m+1}\,,
\]
which completes the proof.
\end{proof}

Define a binary operation $\star$ on $\bbz^2$, that is, a function from $\bbz^2\times\bbz^2$ to $\bbz^2$ as follows:
\begin{equation}\label{gaussian.defstar}
 (i,j)\star(k,l):=(i\wedge j-k\wedge l,k\vee l-i\vee j),\,i,j,k,l\in\bbz\,.
\end{equation}
Fix
\[
 u:=(u_1,\ldots,u_{m-1}),v:=(v_1,\ldots,v_{m-1})\in\bbz^{m-1}\,.
\]
For $n\ge1$, let $V_n(u,v)$ denote the set of all tuples $i$ in $\{1,\ldots,n\}^{2m}$ for which there exists an almost pairing $\pi$ of $\{1,\ldots,2m\}$ and an onto function $\phi$ from $W:=\{1\le j\le 2m:\pi(j)\neq j\}$ to $\{1,\ldots,m-1\}$ such that for all $j\in W$,
\begin{equation}\label{gaussian.eq5}
\phi(j)=\phi(\pi(j))\,,
\end{equation}
and
\begin{equation}\label{gaussian.eq6}
 (i_{j-1},i_j)\star(i_{\pi(j)-1},i_{\pi(j)})=(u_{\phi(j)},v_{\phi(j)})\,.
\end{equation}

\begin{lemma}\label{gaussian.l3}
 There exists a finite constant $C(m)$ depending only on $m$ such that
\begin{equation}\label{gaussian.eq2}
 \#V_n(u,v)\le C(m)n^{m+1},\,n\ge1\,.
\end{equation}
\end{lemma}

\begin{proof}
 Since there are only finitely many almost pair matched words of length $2m$ and each of them has finitely many offspring words, there are only finitely many almost pairings $\pi$, and given any $\pi$ the number of functions $\phi$ satisfying \eqref{gaussian.eq5} and \eqref{gaussian.eq6} is also finite, the proof follows from the conclusion of Lemma \ref{gaussian.l2}.
\end{proof}

\section{Proof of the main result}\label{sec:gaussian}
For the proof of Theorem \ref{gaussian.t1}, we shall need a few notations. Fix $N\in\bbn$. Call a $2m$-tuple $i:=(i_1,\ldots,i_{2m})\in\bbn^{2m}$ ``$N$-Catalan corresponding to $\sigma$'' if there exists $\sigma\in NC_2(2m)$ such that whenever $(j,k)\in\sigma$,
\begin{equation}\label{gaussian.eq8}
 |i_{j-1}-i_{k}|\vee|i_j-i_{k-1}|\le N\,.
\end{equation}

For $i,j,k,l\ge1$, say that 
\begin{equation}\label{eq.defsim}
 (i,j)\sim(k,l)
\end{equation}
if
\[
 |(i\wedge j)-(k\wedge l)|\vee|(i\vee j)-(k\vee l)|\le N\,.
\]
Say that a $(2m)$-tuple $(i_1,\ldots,i_{2m})$ is ``$N$-pair matched'' if there exists a pairing $\pi$ of $\{1,\ldots,2m\}$  such that 
\[
 (i_{j-1},i_j)\sim(i_{\pi(j-1)},i_{\pi(j-1)+1})\,,\mbox{ for }j=1,\ldots,2m\,,
\]
where $\pi(0):=\pi(2m)$. We shall suppress the ``$N$'' in $N$-Catalan and $N$-pair matched if the $N$ of interest is clear from the context.

\begin{lemma}\label{gaussian.l5}
 Fix $N,m,n\ge1$ and $\sigma\in NC_2(2m)$. Let $V_1,\ldots,V_{m+1}$ denote the blocks of the Kreweras complement of $\sigma$. Write
\[
 V_u=\{v_1^u,\ldots,v_{l_u}^u\},\,u=1,\ldots,m+1\,,
\]
where
\begin{equation}\label{gaussian.order}
 v_1^u\le\ldots\le v_{l_u}^u\,.
\end{equation}
Then $i:=(i_1,\ldots,i_{2m})\in\{1,\ldots,n\}^{2m}$ is a $N$-Catalan tuple corresponding to $\sigma$ if and only if there exists $k:=(k_1,\ldots,k_{2m})\in S(\sigma,N)$ where
\begin{equation}\label{eq.defsigmaN}
 S(\sigma,N):=\left\{k\in\{-N,\ldots,N\}^{2m}:\sum_{j=1}^{l_s}k_{v^s_j}=0,\,s=1,\ldots,k+1\right\}\,,
\end{equation}
and a $0$-Catalan tuple $j:=(j_1,\ldots,j_{2m})\in\{1,\ldots,n\}^{2m}$ such that
\begin{equation}\label{gaussian.eq9}
 i_{v_x^u}=j_{v_x^u}+\sum_{w=1}^xk_{v_w^u},\,x=1,\ldots,l_u,\,u=1,\ldots,m+1\,.
\end{equation}
Furthermore, the $j$ and $k$ satisfying \eqref{gaussian.eq9} are unique.
\end{lemma}

\begin{proof}
It is clear that $i:=(i_1,\ldots,i_{2m})\in\bbn^{2m}$ is a $0$-Catalan tuple corresponding to $\sigma$ if and only if 
\[
 i_{v_1^u}=\ldots=i_{v_{l_u}^u},\,u=1,\ldots,m+1\,.
\]
In view of the ordering \eqref{gaussian.order}, similar reasoning as that leading to the above equivalence will yield that $i$ is a $N$-Catalan tuple corresponding to $\sigma$ if and only if
\begin{equation}\label{gaussian.eq10}
 \bigl|i_{v_1^u}-i_{v_2^u}\bigr|\vee\bigl|i_{v_2^u}-i_{v_3^u}\bigr|\vee\ldots\vee\bigl|i_{v_{l_u}^u}-i_{v_1^u}\bigr|\le N,\,u=1,\ldots,m+1\,.
\end{equation}

Now, suppose that $i$ is a $N$-Catalan tuple corresponding to $\sigma$. Define
\[
 k_{v_w^u}:=i_{v_{w+1}^u}-i_{v_w^u},\,w=1,\ldots,l_u,\,u=1,\ldots,m+1\,,
\]
where $v_{l_u+1}^u:=v_1^u$ for $u=1,\ldots,m+1$. It is easy to see because of \eqref{gaussian.eq10} that $k:=(k_1,\ldots,k_{2m})$ thus defined, belongs to $S(\sigma,N)$. Define
\[
 j_{v_w^u}:=i_{v_{l_u}^u},\,w=1,\ldots,l_u,\,u=1,\ldots,m+1\,.
\]
Then, clearly \eqref{gaussian.eq9} holds, and from the equivalence mentioned at the beginning of this proof, it is easy to see that $j:=(j_1,\ldots,j_{2m})$ is a $0$-Catalan word corresponding to $\sigma$. This completes the proof of the ``only if'' part. 

For the ``if'' part, let $j$ and $k$ be given, and define $i$ by \eqref{gaussian.eq9}. Then, \eqref{gaussian.eq10} is immediate, and by the equivalence mentioned just above that equation, it follows that $i$ is a $N$-Catalan tuple.

Finally, for the uniqueness, assume that $i$ is given. If $j$ and $k$ satisfy \eqref{gaussian.eq9}, then from the fact that
\[
 \sum_{w=1}^{l_u}k_{v_w^u}=0\,,
\]
it follows that
\[
 j_{v_{l_u}^u}=i_{v_{l_u}^u}\,,
\]
This, along with the fact that $j$ is $0$-Catalan, specifies $j$. Then, \eqref{gaussian.eq9} determines $k$. This completes the proof.
\end{proof}

\begin{lemma}\label{gaussian.l8}
 Let $\sigma\in NC_2(2m)$, $j$ be a $0$-Catalan tuple corresponding to $\sigma$, $k\in S(\sigma,N)$ and $i$ be given by \eqref{gaussian.eq9}. Assume furthermore that
\begin{equation}\label{gaussian.eq22}
 \#\{\mbox{distinct numbers in }(j_1,\ldots,j_{2m})\}=m+1\,,
\end{equation}
and that
\begin{equation}\label{gaussian.eq23}
 \min\{|j_u-j_v|:1\le u,v\le2m,\,j_u\neq j_v\}>4mN\,.
\end{equation}
Then, for all $(u,v)\in\sigma$,
\[
 E\left(X_{i_{u-1},i_u}X_{i_{v-1},i_v}\right)=R(k_u,k_v)\,.
\]
\end{lemma}

\begin{proof}
We start by showing that
\begin{eqnarray}
 i_u-i_{v-1}&=&k_u\,,\label{gaussian.eq17}\\
\mbox{and }i_v-i_{u-1}&=&k_v\,.\label{gaussian.eq18}
\end{eqnarray}
Assume without loss of generality that $u<v$. Therefore, $\overline u$ and $\overline{v-1}$ belong to the same block in $K(\sigma)$, and furthermore, the block containing them is a subset of $\{\overline u,\overline{u+1},\ldots,\overline{v-1}\}$. Thus, \eqref{gaussian.eq17} follows from \eqref{gaussian.eq9}. We show \eqref{gaussian.eq18} separately for the cases $u\ge2$ and $u=1$. If $u\ge2$, then $\overline{u-1}$ and $\overline v$ are in the same block of $K(\sigma)$, and furthermore that block does not intersect with $\{\overline u,\overline{u+1},\ldots,\overline{v-1}\}$. This shows \eqref{gaussian.eq18}, once again with the help of \eqref{gaussian.eq9}. If $u=1$, then $\overline{2m}$ and $\overline v$ are in the same block of $K(\sigma)$. Obviously, $\overline{2m}$ has to be the last member of its block. Since $(1,v)\in\sigma$, it follows that $\overline v$ is the first member of the block containing itself and $\overline{2m}$, showing that
\[
 i_v=i_{2m}+k_v=i_0+k_v=i_{u-1}+k_v\,.
\]
This complete the proof of \eqref{gaussian.eq18}

Our next aim is to show that either
\begin{equation}\label{gaussian.eq19}
i_u\vee i_{v-1}<i_{u-1}\wedge i_v\,,
\end{equation}
or
\begin{equation}\label{gaussian.eq20}
 i_u\wedge i_{v-1}>i_{u-1}\vee i_v\,,
\end{equation}
holds. Since $\overline{v-1}$ and $\overline v$ belong to distinct blocks of $K(\sigma)$, \eqref{gaussian.eq22} implies that
\[
 j_v\neq j_{v-1}\,.
\]
This, in conjunction with \eqref{gaussian.eq23} establishes that
\[
 \left|j_v-j_{v-1}\right|>4mN\,.
\]
In view of \eqref{gaussian.eq9}, it follows that
\[
 \left|i_v-j_v\right|\vee\left|i_{u-1}-j_v\right|\vee\left|i_{v-1}-j_{v-1}\right|\vee\left|i_{u}-j_{v-1}\right|\le2mN\,.
\]
If $j_v>j_{v-1}$, then in view of the above two inequalities, it is easy to see that
\[
 i_{u-1}\wedge i_v\ge j_v-2mN>j_{v-1}+2mN\ge i_{v-1}\vee i_u\,,
\]
showing that \eqref{gaussian.eq19} holds. Similarly, if $j_v<j_{v-1}$, then \eqref{gaussian.eq20} holds.

Finally to see the claim of the lemma, assume that \eqref{gaussian.eq20} holds. Then, by stationarity, \eqref{gaussian.eq17} and \eqref{gaussian.eq18}, it follows that
\begin{eqnarray*}
 E\left(X_{i_{u-1},i_u}X_{i_{v-1},i_v}\right)&=&R(-k_v,-k_u)\\
&=&R(k_u,k_v)\,,
\end{eqnarray*}
the second equality following from \eqref{gaussian.eq13}.
It is easy to see that when \eqref{gaussian.eq19} holds, then the claim also holds. This completes the proof.
\end{proof}

\begin{lemma}\label{gaussian.l7} Fix $N\ge0$, $m\ge1$. In what follows, ``pair matched'' and ``Catalan'' mean ``$N$-pair matched'' and ``$N$-Catalan'' respectively.\\
(a) If $\pi\in P(2m)\setminus NC_2(2m)$, then
\[
 \lim_{n\to\infty}n^{-(m+1)}\#\{\mbox{pair matched tuples in }\{1,\ldots,n\}^{2m}\mbox{ corresponding to }\pi\}
\]
\[
 =0\,.
\]
(b) If $\sigma\in NC_2(2m)$, then
\[
 \lim_{n\to\infty}n^{-(m+1)}\#\Bigl[\{\mbox{pair matched tuples in }\{1,\ldots,n\}^{2m}\mbox{ corresponding to }\sigma\}
\]
\begin{equation}\label{gaussian.eq14}
\setminus\{\mbox{Catalan tuples in }\{1,\ldots,n\}^{2m}\mbox{ corresponding to }\sigma\}\Bigr]=0\,.
\end{equation}
(c) If $\sigma\in NC_2(2m)$ and $\pi\in P(2m)\setminus\{\sigma\}$, then
\[
 \lim_{n\to\infty}n^{-(m+1)}\#\Bigl[\{\mbox{Catalan tuples in }\{1,\ldots,n\}^{2m}\mbox{ corresponding to }\sigma\}
\]
\[
 \cap\{\mbox{pair matched tuples in }\{1,\ldots,n\}^{2m}\mbox{ corresponding to }\pi\}\Bigr]=0\,.
\]
\end{lemma}

\begin{proof}
We shall prove the claims only for $N=0$. The generalization to the case when $N\neq0$ is trivial, and follows, for example, by  arguments similar to the ones that allow replacement of $u_l$ and $v_l$ by zero in \eqref{gaussian.eq3} and \eqref{gaussian.eq4} respectively.
\begin{proof}[Proof of (a)]
Fix $\pi\in P(2m)\setminus NC_2(2m)$. Let $A=a_1\ldots a_{2m}$ be a word of length $2m$ such that for $j<k$, $a_j=a_k$ if and only if $(j,k)\in\pi$. That is, $A$ is a pair matched word. It is easy to see that any $0$-pair matched tuple is actually an offspring word (considering the entries to be letters) of $A$.  Since $A$ is not a Catalan word, by Lemma \ref{comb.l0} (a), it follows that for all offspring word $B$ of $A$,
\begin{equation}\label{gaussian.l7.eq1}
 \#B\le m\,.
\end{equation}
This completes the proof of (a).
\end{proof}
\begin{proof}[Proof of (b)]
Let $A=a_1\ldots a_{2m}$ be a word of length $2m$ such that for $j<k$, $a_j=a_k$ if and only if $(j,k)\in\sigma$. It is easy to see that a $0$-pair matched tuple which is not a $0$-Catalan tuple generates an offspring word $B=b_1\ldots b_{2m}$ such that at least one of \eqref{comb.l0.eq2} or \eqref{comb.l0.eq3} is violated for some $(j,k)\in\sigma$. Therefore, by Lemma \ref{comb.l0} (c), \eqref{gaussian.l7.eq1} follows, thus proving (b).
\end{proof}
\begin{proof}[Proof of (c)]
Once again, we prove this for $N=0$. If $\pi\notin NC_2(2m)$, then the claim follows by part (a) which has already been proved. So assume that $\pi\in NC_2(2m)$. Let $A_1$ and $A_2$ be Catalan words of length $2m$ corresponding to $\sigma$ and $\pi$ respectively, as in the proof of (a). Since $\pi\neq\sigma$, $A_1$ and $A_2$ are distinct, that is neither of them can be obtained from the other by relabeling letters. Lemma \ref{comb.l0} (b) implies that if $B$ is an offspring of both $A_1$ and $A_2$, then \eqref{gaussian.l7.eq1} holds, thereby establishing (c).
\end{proof}
Since all the claims have been established, this completes the proof of the lemma.
\end{proof}

\begin{remark}
 The number on the left hand side of \eqref{gaussian.eq14} is not necessarily zero. For example, $(1,3,2,1,3,2)$ is $1$-pair matched but not $1$-Catalan corresponding to $\{(1,6),(2,5),(3,4)\}$.
\end{remark}

\begin{proof}[Proof of Theorem \ref{gaussian.t1}]Our first task is to show the existence of a probability measure $\mu$ whose odd moments are zero and the $2m$-th moment is $\beta_{2m}$ for all $m\ge1$. To that end, define the expected ESD $\hat\mu_n$ of $A_n/\sqrt n$ as
\[
 \hat\mu_n(B):=\frac1n\sum_{j=1}^nP(\lambda_j/\sqrt n\in B),\,n\ge1\,,
\]
for all Borel sets $B$. Clearly,
\[
 \int x^m\hat\mu_n(dx)=n^{-(m/2+1)}E\left[\Tr(A_n^{m})\right],\,m,n\ge1\,,
\]
which is zero if $m$ is odd.
If it can be shown that for $m\ge1$,
\begin{equation}\label{gaussian.eq1}
 \lim_{n\to\infty}n^{-(m+1)}E\left[\Tr(A_n^{2m})\right]=\beta_{2m}\,,
\end{equation}
then existence of $\mu$ will follow. In addition, the above is also a significant step in proving that $\mu_n$ converges in probability to $\mu$. We shall come to that a moment later. Before that let us quickly dispose off the issue of uniqueness of $\mu$. It will be shown in Section \ref{sec:stieltjes} that $\beta_{2m}$ are the moments of a compactly supported probability measure which automatically ensures uniqueness; see Remark \ref{rem:bounded}. However, for the sake of completeness, we provide a quick proof of uniqueness via Carleman's condition. In view of Carleman's condition, it suffices to show that
\begin{equation}\label{carleman}
 \sum_{m=1}^\infty\beta_{2m}^{-1/2m}=\infty\,.
\end{equation}
By \eqref{eq.ub}, it follows that
\begin{eqnarray*}
 \sum_{m=1}^\infty\beta_{2m}^{-1/2m}&\ge&{\bar R}^{-1/2}\sum_{m=1}^\infty\left(m!\,\,\#NC_2(2m)\right)^{-1/2m}\\
&\ge&\frac12{\bar R}^{-1/2}\sum_{m=1}^\infty\left(m!\right)^{-1/2m}\\
&\ge&\frac12{\bar R}^{-1/2}\sum_{m=1}^\infty m^{-1}\,,
\end{eqnarray*}
the inequality in the last line following from the fact that $m!\le m^{2m}$ for all $m\ge1$. This establishes \eqref{carleman}. Consequently, there is at most one measure $\mu$ whose odd moments vanish and the $2m$-th moment is $\beta_{2m}$. Thus, to complete the proof, we need to show \eqref{gaussian.eq1} and that
\begin{equation}\label{eq.var}
  \lim_{n\to\infty}n^{-2(m+1)}\Var\left[\Tr(A_n^{2m})\right]=0\,.
\end{equation}

 We now proceed towards showing \eqref{gaussian.eq1}.
Recall that
\[
 \Tr(A_n^{2m})=\sum_{i_1,\ldots,i_{2m}=1}^nX_{i_1,i_2}\ldots X_{i_{2m-1},i_{2m}}X_{i_{2m},i_1}\,,
\]
and therefore
\[
 E\left[\Tr(A_n^{2m})\right]=\sum_{i\in\{1,\ldots,n\}^{2m}}E_i\,,
\]
where
\[
 E_i:=E\left[\prod_{j=1}^{2m}X_{i_{j-1},i_j}\right],\,i\in\bbz^{2m}\,,
\]
the convention being that for $i=(i_1,\ldots,i_{2m})\in\bbz^{2m}$,
$i_0:=i_{2m}$.
Recall the definition of $S(\sigma,N)$ from \eqref{eq.defsigmaN} for $\sigma\in NC_2(2m)$.
Set
\[
 \beta_{2m}^{(N)}:=\sum_{\sigma\in NC_2(2m)}\sum_{k\in S(\sigma,N)}\prod_{(u,v)\in\sigma}R(k_u,k_v),\,m,N\ge1\,.
\]
In view of \eqref{eq.ub}, it follows that
\[
 \lim_{N\to\infty}\beta_{2m}^{(N)}=\beta_{2m}\,.
\]
Fix $\vep>0$. Let $N$ be such that
\[
 |R(u,v)|\le\vep
\]
whenever $u\wedge v\ge N$, and
\begin{equation}\label{gaussian.t1.eq1}
 \left|\beta_{2m}^{(N)}-\beta_{2m}\right|\le\vep\,.
\end{equation}
Recall the definitions of ``$\sim$'', $N$-pair matched and $N$-Catalan words from \eqref{eq.defsim} and the text following it. Let $PM$ denote set of the $N$-pair matched tuples in $\{1,\ldots,n\}^{2m}$, and let $NPM$ denote the corresponding set for the ones that are not $N$-pair matched. 
Write
\begin{eqnarray*}
 E[\Tr(A_n^{2m})]&=&\sum_{i\in PM}E_i+\sum_{i\in NPM}E_i\\
&=:&T_1+T_2\,.
\end{eqnarray*}
We shall apply Wick's formula for estimating $E_i$. Fix $i\in NPM$. Recalling that $P(2m)$ is the set of all pair partitions of $\{1,\ldots,2m\}$, given $\pi\in P(2m)$ there exists $(u,v)\in\pi$ such that
\[
 (i_{u-1},i_u)\not\sim(i_{v-1},i_v)\,.
\]
Therefore, every $\pi\in P(2m)$ can be written as
\begin{equation}\label{eq.pi}
 \pi=\{(u_1^\pi,v_1^\pi),\ldots,(u_m^\pi,v_m^\pi)\}\,,
\end{equation}
where
\begin{equation}\label{eq.notsim}
 (i_{u^\pi_1-1},i_{u_1^\pi})\not\sim(i_{v_1^\pi-1},i_{v_1^\pi})\,.
\end{equation}
Notice that
\begin{equation}\label{eq.connectstar}
 E[X_{i,j}X_{k,l}]=R((i,j)\star(k,l))\,,
\end{equation}
where $\star$ is as defined in \eqref{gaussian.defstar}. By Wick's formula, it follows that
\begin{eqnarray*}
 |E_i|&\le&\vep\sum_{\pi\in P(2m)}\prod_{j=2}^m\left|E\left[X_{i_{u^\pi_j-1},i_{u^\pi_j}}X_{i_{v^\pi_j-1},i_{v^\pi_j}}\right]\right|\\
&=&\vep\sum_{\pi\in P(2m)}\prod_{j=2}^m\left|R((i_{u^\pi_j-1},i_{u^\pi_j})\star(i_{v^\pi_j-1},i_{v^\pi_j}))\right|\,,
\end{eqnarray*}
the equality following by \eqref{eq.connectstar}. Therefore,
\begin{eqnarray}
 |T_2|&\le&\sum_{u,v\in\bbz^{m-1}}\vep\#V_n(u,v)\prod_{j=1}^{m-1}|R(u_j,v_j)|\,,\nonumber\\
&\le&\vep C(m)n^{m+1}\bar R^{m-1}\,,\label{gaussian.neweq6}
\end{eqnarray}
where $V_n$ and $C(m)$ are as in \eqref{gaussian.eq2} and $\bar R$ is as in \eqref{eq.defrbar}, the last inequality following by Lemma \ref{gaussian.l3}. Thus,
\begin{equation}\label{gaussian.eq11}
 \limsup_{n\to\infty}{n^{-(m+1)}}|T_2|\le\vep C(m)\bar R^{m-1}\,.
\end{equation}

We now work with $T_1$. For $\sigma\in NC_2(2m)$, let
\[
 CT(\sigma):=\{\mbox{Catalan tuples in }\{1,\ldots,n\}^{2m}\mbox{ corresponding to }\sigma\}\,,
\]
\begin{eqnarray*}
 CT^\prime(\sigma)&:=&CT(\sigma)\setminus\Biggl(\bigcup_{\pi\in P(2m)\setminus\{\sigma\}}\{\mbox{Pair matched tuples in }\{1,\ldots,n\}^{2m}\\
&&\,\,\,\,\,\,\,\,\,\,\,\,\,\,\,\,\mbox{ corresponding to }\pi\}\Biggr)\,,
\end{eqnarray*}
and
\[
 NCT:=PM\setminus\Biggl(\bigcup_{\sigma\in NC_2(2m)}CT^\prime(\sigma)\Biggr)\,.
\]
Clearly,
\begin{eqnarray*}
 &&PM\setminus\Biggl(\bigcup_{\sigma\in NC_2(2m)}CT(\sigma)\Biggr)\\
&\subset&\Biggl(\bigcup_{\pi\in P(2m)\setminus NC_2(2m)}\{\mbox{Pair matched tuples corresponding to }\pi\}\Biggr)\\
&&\cup\Biggl(\bigcup_{\sigma\in NC_2(2m)}\Bigl[\{\mbox{Pair matched tuples corresponding to }\sigma\}\\
&&\,\,\,\,\,\,\,\,\,\,\,\,\,\,\,\,\setminus\{\mbox{Catalan tuples corresponding to }\sigma\}\Bigr]\Biggr)\,.
\end{eqnarray*}
By Lemma \ref{gaussian.l7} (a) and (b) respectively, the cardinality of the first and the second set on the right hand side is $o(n^{m+1})$.
Also,
\begin{eqnarray*}
 &&\Biggl(\bigcup_{\sigma\in NC_2(2m)}CT(\sigma)\Biggr)\setminus\Biggl(\bigcup_{\sigma\in NC_2(2m)}CT^\prime(\sigma)\Biggr)\\
&\subset&\bigcup_{\sigma\in NC_2(2m)}\Biggl(CT(\sigma)\setminus CT^\prime(\sigma)\Biggr)\,.
\end{eqnarray*}
By Lemma \ref{gaussian.l7} (c), it follows that the cardinality of the set on the right hand side is $o(n^{m+1})$. All the above put together imply that
\begin{equation}\label{gaussian.eq12}
 \lim_{n\to\infty}n^{-(m+1)}\#NCT=0\,.
\end{equation}
Clearly, by definition, if $\sigma,\sigma^\prime\in NC_2(2m)$ and $\sigma\neq\sigma^\prime$, then
\[
 CT^\prime(\sigma)\cap CT^\prime(\sigma^\prime)=\phi\,.
\]
Therefore,
\begin{eqnarray*}
 T_1&=&\sum_{\sigma\in NC_2(2m)}\sum_{i\in CT^\prime(\sigma)}E_i+\sum_{i\in NCT}E_i\\
&=:&T_{11}+T_{12}\,.
\end{eqnarray*}
By Wick's formula, it follows that
\[
 |E_i|\le(2m)!\mbox{ for all }i\in\bbn^{2m}\,,
\]
which in view of \eqref{gaussian.eq12}, shows that
\begin{equation}\label{gaussian.eq15}
 \lim_{n\to\infty}n^{-(m+1)}T_{12}=0\,.
\end{equation}
Fix $\sigma\in NC_2(2m)$ and $i\in CT^\prime(\sigma)$. By definition of $CT^\prime(\sigma)$, it is easy to see that
every $\pi\in P(2m)\setminus\{\sigma\}$ can be written as \eqref{eq.pi} such that \eqref{eq.notsim} holds. Therefore,
\begin{eqnarray}\label{gaussian.neweq4}
 &&\left|E_i-\prod_{(u,v)\in\sigma}E\left[X_{i_u-1,i_u}X_{i_v-1,i_v}\right]\right|\\
\nonumber&\le&\vep\sum_{\pi\in P(2m)\setminus\{\sigma\}}\prod_{j=2}^m\left|R((i_{u^\pi_j-1},i_{u^\pi_j})\star(i_{v^\pi_j-1},i_{v^\pi_j}))\right|\,.
\end{eqnarray}
By Lemma \ref{gaussian.l5}, there exist a $0$-Catalan tuple $j$ and a $k\in S(\sigma,N)$ satisfying \eqref{gaussian.eq9}. Fix $k\in S(\sigma,N)$, $n>4mN$ and define the sets
\begin{eqnarray*}
B_1(k)&=&\Bigl\{i\in\bbz^{2m}:\mbox{\eqref{gaussian.eq9} holds for some }0\mbox{-Catalan tuple }\\
&&\,\,\,\,\,\,\,\,\,\,j\in\{1,\ldots,n\}^{2m}\mbox{ corresponding to }\sigma\Bigr\}\,,\\
B_2(k)&=&\Bigl\{i\in\{1,\ldots,n\}^{2m}:\mbox{\eqref{gaussian.eq9} holds for some }0\mbox{-Catalan tuple }\\
&&\,\,\,\,\,\,\,\,\,\,j\in\{1,\ldots,n\}^{2m}\mbox{ corresponding to }\sigma\Bigr\}\,,\\
B_3(k)&=&\Bigl\{i\in\bbz^{2m}:\mbox{\eqref{gaussian.eq9} holds for some }0\mbox{-Catalan tuple }\\
&&\,\,\,\,\,\,\,\,\,\,j\in\{4mN+1,\ldots,n-4mN\}^{2m}\mbox{ corresponding to }\sigma\Bigr\}\,,\\
B_4(k)&=&\Bigl\{i\in CT^\prime(\sigma):\mbox{\eqref{gaussian.eq9} holds for some }0\mbox{-Catalan tuple }\\
&&\,\,\,\,\,\,\,\,\,\,j\in\{1,\ldots,n\}^{2m}\mbox{ corresponding to }\sigma\Bigr\}\,,\\
B_5(k)&=&\Bigl\{i\in CT^\prime(\sigma):\mbox{\eqref{gaussian.eq9}, \eqref{gaussian.eq22} and \eqref{gaussian.eq23} hold for some }0\mbox{-Catalan }\\
&&\,\,\,\,\,\,\,\,\,\,\mbox{tuple }j\in\{1,\ldots,n\}^{2m}\mbox{ corresponding to }\sigma\Bigr\}\,.
\end{eqnarray*}
Close inspection reveals that
\[
B_3(k)\subset B_2(k)\subset B_1(k)\,, 
\]
and
\begin{eqnarray*}
\#B_1(k)&=&n^{m+1}\,,\\
\#B_3(k)&=&(n-8mN)^{m+1}\,.
\end{eqnarray*}
Therefore,
\begin{equation}\label{gaussian.neweq1}
 \lim_{n\to\infty}n^{-(m+1)}\#B_2(k)=1\,.
\end{equation}
Lemma \ref{gaussian.l5} asserts that
\begin{eqnarray}
CT(\sigma)&=&\bigcup_{k\in S(\sigma,N)}B_2(k)\,,\nonumber\\
CT^\prime(\sigma)&=&\bigcup_{k\in S(\sigma,N)}B_4(k)\,.\label{gaussian.neweq7}
\end{eqnarray}
An outcome of the above is that
\begin{equation}\label{gaussian.neweq2}
 B_2(k)\setminus B_4(k)\subset CT(\sigma)\setminus CT^\prime(\sigma)\,.
\end{equation}
By Lemma \ref{gaussian.l7} (c), it follows that 
\[
 \lim_{n\to\infty}n^{-(m+1)}\#\left[CT(\sigma)\setminus CT^\prime(\sigma)\right]=0\,,
\]
which along with \eqref{gaussian.neweq1} and \eqref{gaussian.neweq2} show that
\[
 \lim_{n\to\infty}n^{-(m+1)}\#B_4(k)=1\,.
\]
Elementary combinatorics shows that
\begin{equation}\label{gaussian.neweq5}
 \#[B_4(k)\setminus B_5(k)]=o(n^{m+1})\,.
\end{equation}
Therefore,
\begin{equation}\label{gaussian.neweq3}
 \lim_{n\to\infty}n^{-(m+1)}\#B_5(k)=1\,.
\end{equation}
By Lemma \ref{gaussian.l8}, it follows that
\[
 \prod_{(u,v)\in\sigma}E\left[X_{i_u-1,i_u}X_{i_v-1,i_v}\right]=\prod_{(u,v)\in\sigma}R(k_u,k_v),\,i\in B_5(k)\,.
\]
Similar arguments as those leading to \eqref{gaussian.neweq6}, in view of \eqref{gaussian.neweq4} and the fact that the family of sets $(B_5(k):k\in S(\sigma,N))$ are disjoint, now establish
\[
 \left|\sum_{k\in S(\sigma,N)}\sum_{i\in B_5(k)}E_i-\sum_{k\in S(\sigma,N)}\#B_5(k)\prod_{(u,v)\in\sigma}R(k_u,k_v)\right|\le \vep C(m)n^{m+1}{\bar R}^{m-1}\,,
\]
where $C(m)$ and $\bar R$  are as in \eqref{gaussian.eq2} and \eqref{eq.defrbar} respectively.
From here, the fact that $S(\sigma,N)$ is a finite set, and that \eqref{gaussian.neweq3} holds for all $k\in S(\sigma,N)$, imply that
\[
 \limsup_{n\to\infty}\left|n^{-(m+1)}\sum_{k\in S(\sigma,N)}\sum_{i\in B_5(k)}E_i-\sum_{k\in S(\sigma,N)}\prod_{(u,v)\in\sigma}R(k_u,k_v)\right|\le
\]
\[
 \vep C(m){\bar R}^{m-1}\,.
\]
Equation \eqref{gaussian.neweq5} allows us to replace $B_5(k)$ by
$B_4(k)$ in the above equation which along with
\eqref{gaussian.neweq7} and the observation that the sets on the
right hand side are pairwise disjoint, lead us to
\begin{equation}\label{var.eq4}
 \limsup_{n\to\infty}\left|n^{-(m+1)}\sum_{i\in CT^\prime(\sigma)}E_i-\sum_{k\in S(\sigma,N)}\prod_{(u,v)\in\sigma}R(k_u,k_v)\right|
\end{equation}
\[
 \le \vep C(m){\bar R}^{m-1}\,.
\]
Adding over $\sigma\in NC_2(2m)$ yields that
\[
 \limsup_{n\to\infty}\left|n^{-(m+1)}T_{11}-\beta^{(N)}_{2m}\right|\le \vep C(m){\bar R}^{m-1}\#NC_2(2m)\,.
\]
Since $\vep$ is arbitrary, the above in view of \eqref{gaussian.t1.eq1}, \eqref{gaussian.eq11} and \eqref{gaussian.eq15} complete the proof of \eqref{gaussian.eq1}.

To complete the proof, \eqref{eq.var} needs to be shown,
or equivalently, in view of  \eqref{gaussian.eq1},
\begin{equation}\label{gaussian.var}
 \lim_{n\to\infty}n^{-2(m+1)}E\left[\left\{\Tr(A_n^{2m})\right\}^2\right]=\beta_{2m}^2\,.
\end{equation}
Notice that
\begin{eqnarray*}
 E\left[\left\{\Tr(A_n^{2m})\right\}^2\right]&=&\sum_{i\in\{1,\ldots,n\}^{4m}}E\left[\prod_{j=1}^{4m}X_{i_{\gamma(j)},i_j}\right]\,,
\end{eqnarray*}
where
\[
 \gamma(j):=\left\{\begin{array}{ll}
                    j-1,&j\in\{1,\ldots,4m\}\setminus\{1,2m+1\}\,,\\
		    2m,&j=1\,,\\
		    4m,&j=2m+1\,.
                   \end{array}
\right.
\]
Denote for $i\in\{1,\ldots,n\}^{4m}$,
\[
E_i:=E\left[\prod_{j=1}^{4m}X_{i_{\gamma(j)},i_j}\right]\,.
\]
Once again, the above expectation can be computed via Wick's formula. Therefore, a similar combinatorial analysis as that for the expected trace goes through, except that now offspring words are replaced by compound offspring words. A sketch of the proof is given below.

Fix $N\ge1$, and say that $i\in\{1,\ldots,n\}^{4m}$ is $N$-pair matched if there exists a pairing $\pi$ of $\{1,\ldots,4m\}$ such that
\[
 \left|i_j\wedge i_{\gamma(j)}-i_{\pi(j)}\wedge i_{\gamma(\pi(j))}\right|\vee\left|i_j\vee i_{\gamma(j)}-i_{\pi(j)}\vee i_{\gamma(\pi(j))}\right|\le N,\,j=1,\ldots,4m\,.
\]
Let $PM$ and $NPM$ denote the sets of $N$-pair matched tuples and non-$N$-pair matched tuples respectively, in $\{1,\ldots,n\}^{4m}$ with this new definition of ``pair matched''.
By arguments similar to those leading to \eqref{gaussian.eq11}, Lemma \ref{comb.l2} now playing the role of Lemma \ref{gaussian.l1}, it follows that
\begin{equation}\label{var.eq1}
 \lim_{N\to\infty}\limsup_{n\to\infty}n^{-(m+1)}\left|\sum_{i\in NPM}E_i\right|=0\,.
\end{equation}
Say that $i\in\{1,\ldots,n\}^{4m}$ is Catalan($\sigma_1,\sigma_2$) for some $\sigma_1,\sigma_2\in NC_2(2m)$ if $(i_1,\ldots,i_{2m})$ is $N$-Catalan with respect to $\sigma_1$, and $(i_{2m+1},\ldots,i_{4m})$ is $N$-Catalan with respect to $\sigma_2$. By Lemma \ref{comb.l1}, it follows that
\[
 \lim_{n\to\infty}n^{-(m+1)}\#\left(PM\bigtriangleup\left(\bigcup_{\sigma_1,\sigma_2\in NC_2(2m)} CT(\sigma_1,\sigma_2)\right)\right)=0\,,
\]
where
\[
CT(\sigma_1,\sigma_2):=\left\{i\in\{1,\ldots,n\}^{4m}:i\mbox{ is Catalan}(\sigma_1,\sigma_2)\right\},\,\sigma_1,\sigma_2\in NC_2(2m)\,.
\]
By standard combinatorial arguments, it follows from the above equation that
\begin{equation}\label{var.eq3}
 \lim_{n\to\infty}n^{-(m+1)}\left|\sum_{i\in PM}E_i-\sum_{\sigma_1,\sigma_2\in NC_2(2m)}\,\sum_{i\in CT(\sigma_1,\sigma_2)}E_i\right|=0\,.
\end{equation}
The arguments leading to \eqref{var.eq4}, once again with the aid of Lemma \ref{gaussian.l8}, imply that for $\sigma_1,\sigma_2\in NC_2(2m)$,
\begin{equation}\label{var.eq5}
 \lim_{N\to\infty}\limsup_{n\to\infty}\Biggl|n^{-(m+1)}\sum_{i\in CT(\sigma_1,\sigma_2)}E_i\,\,\,\,\,
\end{equation}
\[
 -\sum_{k^i\in S(\sigma_i,N)}\left(\prod_{(u,v)\in\sigma_1}R(k^1_u,k^1_v)\right)\left(\prod_{(u,v)\in\sigma_2}R(k^2_u,k^2_v)\right)\Biggr|=0\,.
\]
Clearly,
\begin{eqnarray*}
&&\sum_{\sigma_1,\sigma_2\in NC_2(2m)}\,\sum_{k^i\in S(\sigma_i,N)}\left(\prod_{(u,v)\in\sigma_1}R(k^1_u,k^1_v)\right)\left(\prod_{(u,v)\in\sigma_2}R(k^2_u,k^2_v)\right)\\
&=&\prod_{i=1}^2\left(\sum_{\sigma_i\in NC_2(2m)}\sum_{k^i\in S(\sigma_i,N)}\prod_{(u,v)\in\sigma_i}R(k^i_u,k^i_v)\right)\\
&=&\beta_{2m}^2\,.
\end{eqnarray*}
This, in view of \eqref{var.eq1} to \eqref{var.eq5}, establishes \eqref{gaussian.var}, and thus completes the proof.
\end{proof}

\section{The linear process} \label{sec:inv}
In this section, we study the ESD of a random matrix whose entries are generated from a linear process with independent random variables as the input sequence. In particular, let $\{\epsilon_{i,j}:i,j\in\bbz\}$ be independent, mean zero, variance one random variables which satisfy the Pastur condition
\begin{equation}\label{inv.pastur}
\lim_{n\to\infty}\frac{1}{n^2}\sum_{i,j=1}^nE[\epsilon_{i,j}^2\one( |\epsilon_{i,j}|>\vep\sqrt{n})]=0\,\mbox{ for all }\vep>0\,. 
\end{equation}
Let $\{c_{k,l}:k,l\in\bbz\}$ be a collection of deterministic real numbers such that 
\begin{equation}\label{inv.summable}
0< \sum_{k,l\in\bbz} |c_{k,l}|<\infty\,,
\end{equation}
and
\begin{equation}\label{inv.sym}
 c_{k,l}=c_{l,k},\,k,l\in\bbz\,.
\end{equation}
Define
\begin{equation}\label{inv.defZ}
 Z_{i,j}:=\sum_{k,l\in\bbz}c_{k,l}\epsilon_{i-k,j-l},\,i,j\in\bbz\,,
\end{equation}
where the sum on the right hand side converges in $L^2$ because $c_{k,l}$ are square summable, which is a consequence of \eqref{inv.summable}. While the family of random variables $\{Z_{i,j}:i,j\in\bbz\}$ need not be stationary because the distributions of $\epsilon_{i,j}$ are not necessarily identical, it is easy to see that
\begin{eqnarray}
E(Z_{i,j})&=&0,\,i,j\in\bbz\,,\nonumber\\
E(Z_{i,j}Z_{i-u,j+v})&=&\sum_{k,l\in\bbz}c_{k,l}c_{k-u,l+v}=:R(u,v),\,i,j,u,v\in\bbz\,.\label{inv.defR}
\end{eqnarray}
Define the $n\times n$ symmetric random matrix $A_n$ and  $\mu_n$, the ESD of $A_n/\sqrt n$ by \eqref{gaussian.defA_n} and \eqref{gaussian.defmu_n} respectively. 
The assumption \eqref{inv.summable} ensures that
\[
 \sum_{u\in\bbz}\sum_{v\in\bbz}|R(u,v)|\le\sum_{k\in\bbz}\sum_{l\in\bbz}\left[|c_{k,l}|\sum_{u\in\bbz}\sum_{v\in\bbz}|c_{k-u,l+v}|\right]=\left[\sum_{k,l\in\bbz} |c_{k,l}|\right]^2<\infty\,.
\]
Therefore, we can and do define $\beta_{2m}$ by 
\eqref{eq.defbeta}.
Let $\mu$ be the unique probability measure whose odd moments are all zero, and for $m\ge1$, the $2m$-th moment equals $\beta_{2m}$.

The content of this section is the following result.
 
\begin{theorem}\label{inv.main}
Under assumptions \eqref{inv.pastur} to \eqref{inv.sym}, $\mu_n$ converges weakly in probability to $\mu$.
\end{theorem}

\begin{proof}
We split the proof into two parts. For a finite linear process, we
show that the Stieltjes transform of the ESD of a
matrix made up of Gaussian random variables and another with general
entries satisfying \eqref{inv.pastur} are close to each other using
Lindeberg type argument developed in \cite{chatterjee:2005}. For the
second part, we  show that the L\'evy distance between the truncated
linear process and the original process goes to zero as the truncation level
goes to infinity. 

Fix $m\ge 1$ and let 
$$Z_{i,j}^{(m)}=\sum_{k,l=-m}^m c_{k,l}\epsilon_{i-k,j-l} \text{ for } i,j\ge1\,.$$ 
Define 
\[
 A_n^{(m)}:=((Z_{i,j}^{(m)}))_{n\times n},\,n\ge1\,.
\]
We next define a similar random matrix model, but with Gaussian entries. Let $(G_{i,j}:i,j\in\bbz)$ be i.i.d. standard Gaussian, and set
$$Y_{i,j}^{(m)}=\sum_{k,l=-m}^m c_{k,l}G_{i-k,j-l} \text{ for } i,j\ge1\,.$$ 
Denote
\[
 B_n^{(m)}:=((Y_{i,j}^{(m)}))_{n\times n},\,n\ge1\,.
\]
Assumption \eqref{inv.sym} ensures that the matrices $A_n^{(m)}$ and $B_n^{(m)}$ are symmetric.
By the Lindeberg type argument of ``replacing $\epsilon_{i,j}$ by $G_{i,j}$ one at a time'', it can be shown using the Pastur condition \eqref{inv.pastur} that
\begin{equation}\label{inv.neweq1}
\frac1n\left[\Tr\left(\left(zI_n-\frac{A_n^{(m)}}{\sqrt n}\right)^{-1}\right)-\Tr\left(\left(zI_n-\frac{B_n^{(m)}}{\sqrt n}\right)^{-1}\right)\right]\prob0\,,
\end{equation}
as $n\to\infty$, for all $z$ in the complex plane with non-zero imaginary part. The arguments for above are very similar to those in Subsections 2.3 and 2.4 of \cite{chatterjee:2005} and hence are omitted.

It is easy to see that
\[
 R^{(m)}(u,v):=E(Y^{(m)}_{i,j}Y^{(m)}_{i-u,j+v})=\sum_{k,l=-m}^mc_{k,l}c_{k-u,l+v},\,u,v\in\bbz\,,
\]
and as a trivial consequence of \eqref{inv.sym}, it follows that
\[
 R^{(m)}(u,v)=R^{(m)}(v,u)\,.
\]
Clearly, $R^{(m)}(u,v)$ is non-zero for only finitely many $u,v$, and hence Theorem \ref{gaussian.t1} applies, with a scaling because $R^{(m)}(0,0)$ need not be one. By that result, it follows that for fixed $m$, as $n\to\infty$, the ESD of $B_n^{(m)}/\sqrt n$ converges weakly in probability to the probability measure $\mu^{(m)}$ whose odd moments are all zero and for $l\ge1$, the $2l$-th moment is $\beta^{(m)}_{2l}$ defined by
\[
 \beta_{2l}^{(m)}:=\sum_{\sigma\in NC_2(2l)}\sum_{k\in S(\sigma)}\prod_{(u,v)\in\sigma}R^{(m)}(k_u,k_v),\,l\ge1\,,
\]
with $S(\sigma)$ being as in \eqref{eq.defS}. This, along with \eqref{inv.neweq1}, implies that as $n\to\infty$,
\[
  \frac1n\Tr\left(\left(zI_n-\frac{A_n^{(m)}}{\sqrt n}\right)^{-1}\right)\prob\int\frac1{z-x}\mu^{(m)}(dx),\,z\in\bbc\setminus\bbr\,.
\]
Recalling from \eqref{eq.deflevy} the definition of $L$, a restatement of the above is that
\begin{equation}\label{inv.neweq2}
L\left(\mu_n^{(m)},\mu^{(m)}\right)\prob0\,,
\end{equation}
as $n\to\infty$, where $\mu_n^{(m)}$ denotes the ESD of $A_n^{(m)}/\sqrt n$.
Notice that
\[
 \lim_{m\to\infty}R^{(m)}(u,v)=R(u,v),\,u,v\in\bbz\,.
\]
By using \eqref{inv.summable} to interchange limit and sum, it follows that
\[
 \lim_{m\to\infty}\beta_{2l}^{(m)}=\beta_{2l},\,l\ge1\,.
\]
Therefore,
\begin{equation}\label{inv.neweq3}
 \lim_{m\to\infty}L\left(\mu^{(m)},\mu\right)=0\,.
\end{equation}
In view of \eqref{inv.neweq2} and \eqref{inv.neweq3}, to complete the proof of the result, it suffices to show that
\begin{equation}\label{inv.neweq4}
 \lim_{m\to\infty}\limsup_{n\to\infty}E\left[L^3\left(\mu_n^{(m)},\mu_n\right)\right]=0\,,
\end{equation}
recalling that $\mu_n$ is the ESD of $A_n/\sqrt n$.

To that end, we shall use the fact that for $n\times n$ (deterministic) symmetric matrices $C$ and $D$ with ESD
$\nu_C$ and $\nu_D$ respectively,
\[
 L^3(\nu_C,\nu_D)\le\frac1n\Tr\left((C-D)^2\right)\,,
\]
which is a consequence of the Hoffman-Wielandt inequality; see Corollary A.41, page 502 in \cite{bai:silverstein:2010}.
Using this inequality, it is immediate that
\begin{eqnarray*}
E\left[ L^3\left(\mu_n^{(m)},\mu_n\right)\right]&\le&\frac1n E[\Tr[(A_n/\sqrt n-A_n^{(m)}/\sqrt n)^2]\\
&=&\sum_{k,l\in\bbz: |k|\vee |l|>m}c_{k,l}^2\,.
\end{eqnarray*}
The assumption \eqref{inv.summable} of course ensures that $\{c_{k,l}\}$ is square summable, and thus
 establishes \eqref{inv.neweq4}. Combining this with \eqref{inv.neweq2} and \eqref{inv.neweq3} completes the proof.
\end{proof}

\section{Stieltjes transform}\label{sec:stieltjes}
In this section a characterization of the Stieltjes transform of $\mu$, the LSD in Theorems \ref{gaussian.t1} and \ref{inv.main}, is given via a functional equation. As the reader may have
already noticed, in both the above results, $\mu$ is defined via the
correlations $R(u,v)$ which is as in \eqref{gaussian.defR} or
\eqref{inv.defR}. For this section, let $R(\cdot,\cdot)$ be the
correlations of a weakly stationary mean zero variance one process
$(Y_{ij}:i,j\in\bbz)$, that is,
\begin{eqnarray*}
E(Y_{i,j})&=&0,\,i,j\in\bbz\,,\nonumber\\
E(Y_{i,j}^2)&=&1,\,i,j\in\bbz\,,\nonumber\\
E(Y_{i,j}Y_{i-u,j+v})&=:&R(u,v),\,i,j,u,v\in\bbz\,.
\end{eqnarray*}
As in Section \ref{sec:gaussian}, we assume \eqref{eq.symmetry} and \eqref{eq.defrbar}. As before, let $\mu$ be the unique even probability measure whose  $2m$-th moment equals $\beta_{2m}$ which is as defined in \eqref{eq.defbeta}. Recall that the Stieltjes transform of the probability measure $\mu$ on $\bbr$ is denoted by,
$$\mathcal G(z)=\int_{\bbr} \frac1{z-x}\mu(dx), z\in \bbc.$$
The main result of this section is Theorem \ref{theorem-st} below.

Let the Fourier transform of covariance function $\{R(k,l)\}_{k,l\in\bbz}$ be given by
 \begin{equation*}
 f(x,y)= \sum_{k,l\in \bbz} R(k,l) \exp(2\pi i(kx+ly))\text{ for }
(x,y)\in [0,1]\times[0,1].
\end{equation*}
Note that by \eqref{eq.symmetry},  it follows that $f(x,y)$ is a real, symmetric function. For stating the main result, we need the following proposition.

\begin{prop}\label{lem:stieltjes}
 Suppose that ${\mathcal H}_1$ and ${\mathcal H}_2$ are functions from $\bbc\times[0,1]$ to $\bbc$ satisfying the following for $i=1,2$:
\begin{enumerate}
\item for all $x\in[0,1]$ and $z\in\bbc$,
\begin{equation}\label{eq1}
 z{\mathcal H}_i(z,x)=1+{\mathcal H}_i(z,x)\int_0^1{\mathcal H}_i(z,y)f(x,y)dy\,,
\end{equation}
\item there exists a neighborhood $N_i$ (independent of $x$) of infinity such that for all $x\in[0,1]$, ${\mathcal H}_i(\cdot,x)$ is analytic on $N_i$, 
\item for all $x\in[0,1]$,
\begin{equation}\label{eq2}
 \lim_{z\to\infty}z{\mathcal H}_i(z,x)=1\,,
\end{equation}
\item and
\begin{equation}\label{eq3}
 {\mathcal H}(-z,x)=-{\mathcal H}(z,x),\,z\in\bbc,\,x\in[0,1]\,.
\end{equation}
\end{enumerate}
Then
\[
 {\mathcal H}_1\equiv{\mathcal H}_2\mbox{ on }N_1\cap N_2\,.
\]
\end{prop}

\begin{proof}
By the assumption of analyticity on $N_1$ and \eqref{eq3},  for all $x\in[0,1]$ and $k\ge1$, there exist $H_{2k}(x)\in\bbc$ such that
\[
 {\mathcal H}_1(z,x)=\sum_{k=0}^\infty H_{2k}(x)z^{-(2k+1)},\,z\in N_1,\,x\in[0,1]\,.
\]
The condition \eqref{eq2} implies that
\begin{equation}\label{eq4}
 H_0(x)=1\,.
\end{equation}
By comparing the power series expansion of the LHS and the RHS of \eqref{eq1}, one will arrive at the recursion
\begin{equation}\label{eq5}
 H_{2m}(x)=\sum_{k=1}^mH_{2(m-k)}(x)\int_0^1f(x,y)H_{2(k-1)}(y)dy\,.
\end{equation}
Clearly, a power series expansion of ${\mathcal H}_2$ will also satisfy \eqref{eq4} and \eqref{eq5}, and therefore they have to match with that of ${\mathcal H}_1$. This completes the proof.
\end{proof}

The following is the main result.
\begin{theorem}\label{theorem-st}
 There exists a function ${\mathcal H}$ satisfying the assumptions of the Proposition~\ref{lem:stieltjes}. The Stieltjes transform $\mathcal G$ of the LSD $\mu$ is given by
\[
 {\mathcal G}(z):=\left[\int_0^1{\mathcal H}(z,x)dx\right],\,z\in\bbc\,.
\]
\end{theorem}

Before proceeding to the proof, we introduce some notations which will be used for the same. Fix $\sigma\in NC_2(2m)$, and denote its Kreweras complement by $(V_1,\ldots,V_{m+1})$. Although the Kreweras complement is a partition of $\{\overline1,\ldots,\overline{2m}\}$, for the ease of notation, $V_1,\ldots,V_{m+1}$ will be thought of as subsets of $\{1,\ldots,2m\}$, that is, the overline will be suppressed. In order to ensure uniqueness in the notation, we impose the requirement that the blocks $V_1,\ldots,V_{m+1}$ are ordered in the following way. If $1\le i<j\le m+1$, then the {\bf maximal} element of $V_i$ is strictly less than that of $V_j$. Let $\mathcal T_\sigma$ be the unique  function from $\{1,\ldots,2m\}$ to $\{1,\ldots,m+1\}$ satisfying
\[
 i\in V_{\mathcal T_\sigma(i)},\,1\le i\le2m\,.
\]
For example, if
\[
 \sigma:=\{(1,4),(2,3),(5,6)\}\,,
\]
then $\mathcal T_\sigma(1) = 2, \mathcal T_\sigma(2) = 1, \mathcal
T_\sigma(3) = 2, \mathcal T_\sigma(4) = 4, \mathcal T_\sigma(5) = 3,
\mathcal T_\sigma(6) = 4$. 
Define the function $L_\sigma$ from $\bbr^{m+1}$ to $\bbr$ by
\[
 L_\sigma(x):=\prod_{(u,v)\in\sigma}f\left(x_{\mathcal T_\sigma(u)},x_{\mathcal T_\sigma(v)}\right),\,x\in\bbr^{m+1}\,.
\]
Observe that
\begin{equation}\label{stieltjes.eq2}
 \mathcal T_\sigma(2m)=m+1\,.
\end{equation}
Finally, set
\[
 h_\sigma(y):=\int_0^1\ldots\int_0^1L_\sigma(x_1,\ldots,x_m,y)dx_m\ldots dx_1,\,y\in\bbr\,.
\]

The following lemma shows how the moments of the LSD $\mu$ are related to this function. 
\begin{lemma}\label{lemma:alternate}
Let $\beta_{2m}$ be as in \eqref{eq.defbeta}. Then
\begin{equation*}
\beta_{2m}=\sum_{\sigma\in NC_2(2m)}\int_0^1h_\sigma(y)dy,\,m\ge1\,.
\end{equation*}
\end{lemma}

\begin{proof}
 Fix $m\ge1$ and $\sigma\in NC_2(2m)$. 
All that needs to be shown is
\begin{equation}\label{eq.alternate}
\sum_{k\in S(\sigma)}\prod_{(u,v)\in\sigma}R(k_u,k_v)= \int_{[0,1]^{m+1}}L_\sigma(x)\,dx\,,
\end{equation}
where $S(\sigma)$ is as in \eqref{eq.defS}.
Observe that
\[
 \sum_{(u,v)\in\sigma}\left[k_ux_{\mathcal T_\sigma(u)}+k_vx_{\mathcal T_\sigma(v)}\right]=\sum_{u=1}^{2m}k_ux_{\mathcal T_\sigma(u)}=\sum_{l=1}^{m+1}x_l\sum_{j\in V_l}k_j\,,
\]
and hence
\begin{eqnarray*}
&& \int_{[0,1]^{m+1}}L_\sigma(x)\,dx\\
&=&\int_{[0,1]^{m+1}}\left[\sum_{k\in\bbz^{2m}}\exp\left(2\pi i\sum_{l=1}^{m+1}x_l\sum_{j\in V_l}k_j\right)\prod_{(u,v)\in\sigma}R(k_u,k_v)\right]dx\\
&=&\sum_{k\in\bbz^{2m}}\left(\prod_{(u,v)\in\sigma}R(k_u,k_v)\right)\int_{[0,1]^{m+1}}\exp\left(2\pi i\sum_{l=1}^{m+1}x_l\sum_{j\in V_l}k_j\right)dx
\end{eqnarray*}
\begin{eqnarray*}
&=&\sum_{k\in S(\sigma)}\prod_{(u,v)\in\sigma}R(k_u,k_v)\,,
\end{eqnarray*}
the interchange of sum and integral in the second last line being justified by assumption~\eqref{eq.defrbar}. This completes the proof.
\end{proof}

The next two lemmas give a recursive relation for the function $h_\sigma$.

\begin{lemma}\label{stieltjes.l1}
 Assume that $\sigma\in NC_2(2m)$ can be written as 
\begin{equation}\label{stieltjes.l1.eq1}
 \sigma=\sigma_1\cup\sigma_2\,,
\end{equation}
where $\sigma_1\in NC_2(2k)$ for some $1\le k\le m-1$, and $\sigma_2$ is a non-crossing pair partition of $\{2k+1,\ldots,2m\}$. Viewing $\sigma_2$ as an element of $NC_2(2m-2k)$ by the obvious relabeling of $2k+1,\ldots,2m$ to $1,\ldots,2m-2k$ respectively, it is true that
\[
 h_\sigma(y)=h_{\sigma_1}(y)h_{\sigma_2}(y),\,y\in\bbr\,.
\]
\end{lemma}

\begin{proof}It is easy to see from \eqref{stieltjes.eq2} and \eqref{stieltjes.l1.eq1} that
\begin{equation}\label{stieltjes.l1.eq2}
 \mathcal T_\sigma(j)\in\{1,\ldots,k,m+1\},\mbox{ for }1\le j\le2k\,,
\end{equation}
and
\begin{equation}\label{stieltjes.l1.eq3}
 \mathcal T_\sigma(j)\in\{k+1,\ldots,m+1\},\mbox{ for }2k+1\le j\le2m\,.
\end{equation}
Write
\[
 L_\sigma(x)=
\]
\[
\left(\prod_{(u,v)\in\sigma:u,v\le2k}f\left(x_{\mathcal T_\sigma(u)},x_{\mathcal T_\sigma(v)}\right)\right) \left(\prod_{(u,v)\in\sigma:u,v>2k}f\left(x_{\mathcal T_\sigma(u)},x_{\mathcal T_\sigma(v)}\right)\right)\,.
\]
By \eqref{stieltjes.l1.eq2} and \eqref{stieltjes.l1.eq3}, it follows that
\begin{eqnarray*}
&& h_\sigma(x_{m+1})\\
&=&\left(\int_0^1\ldots\int_0^1\prod_{(u,v)\in\sigma:u,v\le2k}f\left(x_{\mathcal T_\sigma(u)},x_{\mathcal T_\sigma(v)}\right)dx_k\ldots dx_1\right)\\
&&\left(\int_0^1\ldots\int_0^1\prod_{(u,v)\in\sigma:u,v>2k}f\left(x_{\mathcal T_\sigma(u)},x_{\mathcal T_\sigma(v)}\right)dx_m\ldots dx_{k+1}\right)
\end{eqnarray*}
\begin{eqnarray*}
&=&h_{\sigma_1}(x_{m+1})h_{\sigma_2}(x_{m+1})\,.
\end{eqnarray*}
This completes the proof.
\end{proof}

\begin{lemma}\label{stieltjes.l2}
 If
\begin{equation*}
 \sigma=\{(1,2m)\}\cup\sigma_1\,,
\end{equation*}
for some non-crossing pair partition $\sigma_1$ of $\{2,\ldots,2m-1\}$, then
\[
 h_\sigma(z)=\int_0^1h_{\sigma_1}(y)f(y,z)dy,\,z\in\bbr\,,
\]
where, once again, $\sigma_1$ is viewed as an element of $NC_2(2m-2)$.
\end{lemma}

\begin{proof}
 Throughout the proof, $\sigma_1$ is to be thought of as an element of \\$NC_2(2m-2)$. Clearly,
\begin{eqnarray*}
\mathcal T_\sigma(1)&=&m\,,\\
\mathcal T_\sigma(2m)&=&m+1\,,\\
\mathcal T_\sigma(j)&=&\mathcal T_{\sigma_1}(j-1),\,2\le j\le2m-1\,.
\end{eqnarray*}
The above equations imply that
\[
 L_\sigma(x)=f(x_m,x_{m+1})L_{\sigma_1}(x_1,\ldots,x_m),\,x\in\bbr^{m+1}\,.
\]
Thus,
\begin{eqnarray*}
 h_\sigma(z)&=&\int_0^1\ldots\int_0^1f(x_m,z)L_{\sigma_1}(x_1,\ldots,x_m)dx_m\ldots dx_1\\
&=&\int_0^1f(x_m,z)\left\{\int_0^1\ldots\int_0^1L_{\sigma_1}(x_1,\ldots,x_m)dx_{m-1}\ldots dx_1\right\}dx_m\\
&=&\int_0^1f(x_m,z)h_{\sigma_1}(x_m)dx_m\,,
\end{eqnarray*}
which completes the proof.
\end{proof}

With the aid of the above lemmas, we proceed to the proof of Theorem~\ref{theorem-st}.

\begin{proof}[Proof of Theorem~\ref{theorem-st}]
We start with defining the functions $$H_0(x)=1, \qquad H_{2m}(x)=\sum_{\sigma\in NC_2(2m)}h_{\sigma}(x).$$  Lemma~\ref{lemma:alternate} implies that
\begin{equation}\label{stieltjes.eq3}
\beta_{2m}=\int_0^1 H_{2m}(x)dx\,. 
\end{equation}
By \eqref{eq.defrbar}, it follows that
$|f(x,y)|\le \overline{R}$ uniformly for all $x$ and $y$ in
  $[0,1]$. 
Lemmas \ref{stieltjes.l1} and \ref{stieltjes.l2} applied inductively imply that
\[
 \sup_{0\le x\le1}|h_\sigma(x)|\le{\bar R}^m,\,\sigma\in NC_2(2m),\,m\ge1\,.
\]
 Since $\#NC_2(2m)\le 4^m$, it holds that
  $$\sup_{0\le x\le1}|H_{2m}(x)|\le {\bar R}^m2^{2m},\,m\ge1\,.$$
Consequently,  
\begin{equation}\label{stieltjes.eq1}
    \sup_{0\le x\le 1}\limsup_{m\to\infty}|H_{2m}^{1/2m}(x)|\le 2{\bar R}^{1/2}<\infty\,.
  \end{equation}
Therefore, the power series
$$  \mathcal H(z,x)=\sum_{m=0}^\infty
  \frac{H_{2m}(x)}{z^{2m+1}}$$
converges on $\{z\in\bbc:|z|>2{\bar R}^{1/2}\}$ for every fixed $x\in[0,1]$. 
Note that this neighborhood around infinity
  is independent of $x\in[0,1]$. It is easy to see that $z\mathcal H(z,x)$ has a power series expansion with
  the leading term as $1$ and hence $z\mathcal{H}(z,x)\to 1$ as
  $|z|\to\infty$.  It follows from the definition of
  $\mathcal{H}(z,x)$ that $\mathcal H(-z,x)=-\mathcal H(z,x)$.
Recall that the Stieltjes transform $\mathcal G$ of $\mu$ satisfies $$ \mathcal G(z)=\sum_{m=0}^{\infty}\beta_{2m}z^{-(2m+1)}\,,$$
(with the obvious convention that $\beta_0:=1$) which yields,
$$\mathcal G(z)=\int_0^1 \mathcal H(z,x) dx.$$
Equation \eqref{eq1} with ${\mathcal H}_i$ replaced by $\mathcal H$ is all that remains to be checked. 

To that end, we derive a recursion for $\mathcal H(z,x)$ using the
  properties of $h_\sigma(x)$.
Recall that there is a natural one-one correspondence between\\ $NC_2(2m)$ and the set of Catalan words of length $2m$ with the understanding that two words will be considered identical if one can be obtained from the other by a relabeling of letters. Keeping this correspondence in mind, by an abuse of notation, we shall now consider $h_w(x)$ for Catalan words $w$, and denote by $NC_2(2m)$ the set of Catalan words of length $2m$.
 Note that any Catalan word $w$
  of length $2m $ can be written as $w=aw_1aw_2$, for some $w_1\in NC_2(2k-2)$
  and $w_2\in NC_2(2m-2k)$.  So if
  $$H_{2m,k}(x):=\sum_{w_1\in NC_2(2k-2)}\sum_{w_2\in NC_2(2m-2k)}h_{aw_1aw_2}(x)\,,$$ then
  $$H_{2m}(x)=\sum_{k=1}^mH_{2m,k}(x).$$
Notice that
  \begin{align*}
    H_{2m,k}(x)
    &=\sum_{w_1\in NC_2(2k-2)}h_{aw_1a}(x)\sum_{w_2\in NC_2(2m-2k)}h_{w_2}(x)
  \end{align*}  
\begin{eqnarray*}
&=&\sum_{w_1\in NC_2(2k-2)}\int_0^1\left[f(x,y)h_{w_1}(y)\right]dy\sum_{w_2\in NC_2(2m-2k)}h_{w_2}(x)\\
&=&\int_0^1\left[f(x,y)H_{2(k-1)}(y) H_{2(m-k)}(x)\right]dy\,,
\end{eqnarray*}
the equalities in the first two lines following from Lemmas \ref{stieltjes.l1} and \ref{stieltjes.l2} respectively.
  As a consequence,
  $$H_{2m}(x)=\sum_{k=1}^mH_{2(m-k)}(x) \int_0^1f(x,y)H_{2(k-1)}(y)dy .$$

  Now by an easy computation it follows that,
$$\mathcal H(z,x)=\sum_{m=0}^{\infty}H_{2m}z^{-(2m+1)}=\frac1z+\frac1z\mathcal{H}(z,x)\int_0^1f(x,y)\mathcal{H}(z,y)dy.$$
Hence,
  \begin{equation*}\label{eq:h:form1}
    z\mathcal H(z,x)=1+\mathcal H(z,x)\int_0^1f(x,y)\mathcal H(z,y)dy.
  \end{equation*}
This completes the proof.
\end{proof}

\begin{remark}\label{rem:bounded}
 Equations \eqref{stieltjes.eq3} and \eqref{stieltjes.eq1} imply that
\[
 \lim_{m\to\infty}\beta_{2m}^{1/2m}<\infty\,,
\]
which implies that the probability measure $\mu$ is compactly supported.
\end{remark}

\section{Special cases and examples}\label{sec:example}

In this section, we attempt to give a better description of the
probability measure $\mu$, which appears as the LSD in Theorems
\ref{gaussian.t1} and \ref{inv.main}, in some special cases. As in Section \ref{sec:stieltjes}, $R(\cdot,\cdot)$ are the
correlations of a weakly stationary mean zero {\bf variance one} process
$(Y_{ij}:i,j\in\bbz)$. As always,  \eqref{eq.symmetry} and \eqref{eq.defrbar} are assumed, and $\mu$ is the unique even probability measure whose  $2m$-th moment equals $\beta_{2m}$ which is as defined in \eqref{eq.defbeta}.
The first main result of this section is the following.

\begin{theorem}\label{example.t1}
 Assume that
\begin{equation}\label{example.eq1}
 R(u,v)=R(u,0)R(0,v),\,u,v\in\bbz\,.
\end{equation}
Then, the function $r(\cdot)$ defined on $[-\pi,\pi]$ by
\[
 r(x):=\sum_{k=-\infty}^\infty R(k,0)e^{-ikx},\,-\pi\le x\le\pi\,,
\]
is a well defined function, that is the sum on the right hand side converges absolutely, and its range is contained in $[0,\infty)$. Furthermore,
\[
 \mu=\mu_r\boxtimes\mu_s\,,
\]
where $\mu_r$ denotes the law of $r(U)$, $U$ being a Uniform $(-\pi,\pi)$ random variable, $\mu_s$ denotes the WSL whose density is given by
\begin{equation}\label{example.eq2}
 \mu_s(dx):=\frac{\sqrt{4-x^2}}{2\pi}\one(|x|\le2)dx\,,
\end{equation}
and `` $\boxtimes$'' denotes the free product convolution.
\end{theorem}

\begin{remark}\label{ext.rem1}
 In \cite{bercovici:voiculescu:1993}, the free multiplicative convolution of two probability measures with possibly unbounded support, {\bf at least one of which is supported on a subset of $[0,\infty)$}, has been defined. Hence, in the above result, the claim that $r(\cdot)$ is non-negative is needed.
\end{remark}

\begin{proof}[Proof of Theorem \ref{example.t1}]
 In view of assumptions \eqref{eq.defrbar} and \eqref{example.eq1}, it is easy to see that
\begin{equation}\label{example.t1.eq1}
 \sum_{k=-\infty}^\infty|R(k,0)|={\bar R}^{1/2}<\infty\,,
\end{equation}
and hence the infinite sum defining $r(x)$ is absolutely convergent. Observing that $(R(k,0):k\in\bbz)$ is the autocovariance function of the one-dimensional process $(Y_{-k,0}:k\in\bbz)$, Corollary 4.3.2, page 120 of \cite{brockwell:davis:1991} implies that
\[
 r(x)\in[0,\infty),\,-\pi\le x\le\pi\,.
\]
This, in view of \eqref{example.t1.eq1}, establishes that the range of $r(\cdot)$ is a compact subset of $[0,\infty)$.

For establishing the other claim, we shall use Theorem 14.4 of \cite{nica:speicher:2006} which applies to compactly supported probability measures. From that result, it follows that $\mu_r\boxtimes\mu_s$ is an even probability measure, and
\begin{equation}\label{eq.moments}
\int x^{2m}(\mu_r\boxtimes\mu_s)(dx)=\sum_{\sigma\in NC_2(2m)}\prod_{j=1}^{m+1}\int x^{l_j^\sigma}\mu_r(dx),\,m\ge1\,, 
\end{equation}
where for any $\sigma\in NC_2(2m)$, $l_1^\sigma,\ldots,l_{m+1}^\sigma$ denote the block sizes of the Kreweras complement of $\sigma$. It is easy to see from the definition of $r(\cdot)$ that
\[
 \int x^j\mu_r(dx)=\sum_{(k_1,\ldots,k_j)\in\bbz^j:k_1+\ldots+k_j=0}\,\,\prod_{i=1}^jR(k_i,0),\,j\ge1\,.
\]
For $\sigma\in NC_2(2m)$ let the notation for its Kreweras complement be as in \eqref{eq.defV}, and recall the definition of $S(\sigma)$ from \eqref{eq.defS}. The above two identities put together imply that for $m\ge1$,
\begin{eqnarray*}
 \int x^{2m}(\mu_r\boxtimes\mu_s)(dx)&=&\sum_{\sigma\in NC_2(2m)}\sum_{k\in S(\sigma)}\prod_{j=1}^{2m}R(k_j,0)\\
&=&\sum_{\sigma\in NC_2(2m)}\sum_{k\in S(\sigma)}\prod_{(u,v)\in\sigma}R(k_u,0)R(k_v,0)\\
&=&\beta_{2m}
\end{eqnarray*}
\begin{eqnarray*}
&=&\int x^{2m}\mu(dx)\,,
\end{eqnarray*}
the equality in the second last line following from \eqref{example.eq1}. This completes the proof.
\end{proof}

The next result is the other main result of this section.

\begin{theorem}\label{example.t2}
 If 
\begin{equation}\label{example.eq3}
 R(k,0)=0\mbox{ for all }k\neq0\,,
\end{equation}
then $\mu=\mu_s$, where $\mu_s$ is the WSL as defined in \eqref{example.eq2}.
\end{theorem}

\begin{proof}
 Clearly, it suffices to show that for all $m\ge1$ and $\sigma\in NC_2(2m)$,
\begin{equation}\label{example.t2.eq1}
 \sum_{k\in S(\sigma)}\prod_{(u,v)\in\sigma}R(k_u,k_v)=1\,,
\end{equation}
where $S(\sigma)$ is as in \eqref{eq.defS}. This is because, if the above is established, then it will follow that
\[
 \beta_{2m}=\#NC_2(2m)=\int x^{2m}\mu_s(dx)\,.
\]
In order to show \eqref{example.t2.eq1}, fix $m\ge1$ and $\sigma\in
NC_2(2m)$. What we shall show is that if $k\in S(\sigma)$ is such that
\begin{equation}\label{example.t2.eq2}
 \prod_{(u,v)\in\sigma}R(k_u,k_v)\neq0\,,
\end{equation}
then,
\begin{equation}\label{example.t2.eq3}
 k_1=\ldots=k_{2m}=0\,.
\end{equation}
Recalling that $R(0,0)=1$, which is a consequence of the assumption that the process $(Y_{ij}:i,j\in\bbz)$ mentioned at the beginning of this section, has variance one, the above will imply \eqref{example.t2.eq1}.

The claim \eqref{example.t2.eq3} is a tautology when $m=1$. As the induction hypothesis, we
assume that for a fixed $m\ge1$ and {\bf all} $\sigma\in NC_2(2m)$,
\eqref{example.t2.eq2} implies \eqref{example.t2.eq3}. To complete the
induction step, fix $\sigma\in NC_2(2m+2)$, and let
\eqref{example.t2.eq2} hold. By the property of non-crossing pair
partition, there exists $j\in\{1,\ldots,2m+1\}$ such that
$(j,j+1)\in\sigma$. Recalling that $K(\sigma)$, the Kreweras
complement of $\sigma$, is the maximal partition $\overline\sigma$ of
$\{\overline 1,\ldots,\overline {2m}\}$ such that
$\sigma\cup\overline\sigma$ is a non-crossing partition of
$\{1,\overline1,\ldots,2m,\overline{2m}\}$, it follows that $(\bar
j)\in K(\sigma)$. Hence,
\[
 k_j=0\,.
\]
Since \eqref{example.t2.eq2} holds, it follows that
\[
 R(k_j,k_{j+1})\neq0\,,
\]
which along with \eqref{example.eq3} implies that
\[
 k_{j+1}=0\,.
\]
If $\bar\sigma$ denotes the element of $NC_2(2m)$ obtained from $\sigma$ by deleting $(j,j+1)$ and the obvious relabeling, then it is easy to see that
\[
 (\bar k_1,\ldots,\bar k_{2m}):=(k_1,\ldots,k_{j-1},k_{j+2},\ldots,k_{2m+2})\in S(\bar\sigma)\,,
\]
and 
\[
 \prod_{(u,v)\in\bar\sigma}R(\bar k_u,\bar k_v)\neq0\,.
\]
By the induction hypothesis, it follows that
\[
 \bar k_1=\ldots=\bar k_{2m}=0\,,
\]
which establishes the induction step, and thereby completes the proof.
\end{proof}

Now, we shall see the relevance of the two main results proved above
in the light of Theorem \ref{inv.main}.

\subsection*{Corollary of Theorem \ref{example.t1}} 

Let $\{\epsilon_{i,j}:i,j\in\bbz\}$ be as in Section \ref{sec:inv}; in
particular, the Pastur condition \eqref{inv.pastur} holds. Let
$\{c_k:k\in\bbz\}$ be a sequence of real numbers such that
\begin{equation}\label{example.neweq1}
\sum_{k=-\infty}^\infty |c_{k}|<\infty\,,
\end{equation}
and
\begin{equation}\label{example.neweq2}
 \sum_{k=-\infty}^\infty c_k^2=1\,.
\end{equation}
Set
\[
 c_{k,l}:=c_kc_l,\,k,l\in\bbz\,.
\]
Define $Z_{i,j}$ and $R(\cdot,\cdot)$ by \eqref{inv.defZ} and \eqref{inv.defR} respectively.  Clearly, \eqref{inv.summable} and \eqref{inv.sym} hold, and the process $(Z_{i,j}:i,j\in\bbz)$ is weakly stationary with mean zero and variance one.
 Also, 
\begin{eqnarray*}
 R(u,v)&=&\left(\sum_kc_kc_{k-u}\right)\left(\sum_lc_lc_{l+v}\right)\\
&=&\left(\sum_k\sum_{k^\prime}c_kc_{k-u}c_{k^\prime}^2\right)\left(\sum_l\sum_{l^\prime}c_lc_{l+v}c_{l^\prime}^2\right)\\
&=&R(u,0)R(0,v)\,,
\end{eqnarray*}
the second equality following from \eqref{example.neweq2}. Let $A_n$ and $\mu_n$ be as in
\eqref{gaussian.defA_n} and \eqref{gaussian.defmu_n} respectively,
that is, the former is the $n\times n$ matrix whose $(i,j)$-th entry
is $Z_{i\wedge j,i\vee j}$, and the latter is the ESD of $A_n/\sqrt
n$.  Let $\mu_r$ and $\mu_s$ be as in the statement of Theorem
\ref{example.t1}. Then, as a corollary of the result mentioned above
and Theorem \ref{inv.main}, it follows that, $\mu_n$ converges weakly
in probability to $\mu_r\boxtimes\mu_s$.

\subsection*{Corollary of Theorem \ref{example.t2}}

Once again, let $\{\epsilon_{i,j}:i,j\in\bbz\}$ be as in Section
\ref{sec:inv} satisfying \eqref{inv.pastur}. Assume that
$\{c_{k,l}:k,l\in\bbz\}\subset\bbr$ is such that \eqref{inv.summable}
and \eqref{inv.sym} hold, and furthermore
\begin{eqnarray}
\sum_{l=-\infty}^\infty c_{k,l}c_{k^\prime,l}&=&\one(k=k^\prime)\mbox{ for all }k,k^\prime\in\bbz\,.\label{eq.orth1}
\end{eqnarray}
As before, let $Z_{i,j}$, $A_n$ and
$\mu_n$ be as in \eqref{inv.defZ}, \eqref{gaussian.defA_n} and
\eqref{gaussian.defmu_n} respectively.
It is easy to see that the conditions imposed above ensure that $(Z_{i,j}:i,j\in\bbz)$ is a mean zero variance one weakly stationary process, and that
\eqref{example.eq3} holds.   Then by Theorem \ref{inv.main}
and Theorem \ref{example.t2}, it follows that $\mu_n$ converges weakly
in probability to $\mu_s$ which is the WSL defined in
\eqref{example.eq2}.

We end this section by revisiting Examples 1 to 4 mentioned in Section
\ref{sec:intro}.

\noindent{\bf Example 1.} To start with, one needs to argue the existence of a stationary centered Gaussian process $\{Z_{i,j}:i,j\in\bbz\}$ satisfying
\[
 E[Z_{0,0}Z_{u,v}]=\rho^{|u|+|v|},\,u,v\in\bbz\,.
\]
That, however, is obvious from the observation that
\[
 \rho^{|u|+|v|}=\int_{(-\pi,\pi]^2}e^{i(ux+vy)}F(dx)F(dy),\,u,v\in\bbz\,,
\]
where $F$ is the spectral measure of the autocovariance function
$(\rho^{|h|}:h\in\bbz)$; see Herglotz theorem (Theorem 4.3.1 in
\cite{brockwell:davis:1991}). By Theorem \ref{example.t1} and results
about the AR($1$) process, it follows that $\mu_n$ converges in
probability to $\mu_r\boxtimes\mu_s$, where $\mu_r$ is the law of
$\frac{1-\rho^2}{1-2\rho\cos U+\rho^2}$, $U$ being an Uniform
$(-\pi,\pi)$ random variable.

\noindent{\bf Example 2.} Notice that the $(i,j)$-th entry of $A_n$ is given by
\[
 (N+1)\sum_{k,l\in\bbz}c_kc_lG_{i-k,j-l}=:(N+1)Y_{i,j}\,,
\]
where $c_k:=(N+1)^{-1/2}\one(-N\le k\le0)$. Then \eqref{example.neweq1} and \eqref{example.neweq2} hold, and therefore, the ESD of $((Y_{i,j}/\sqrt n))_{n\times n}$ converges to $\mu_r\boxtimes\mu_s$, where $\mu_r$ is the law of
\[
 1+2(N+1)^{-2}\sum_{k=1}^N(N-k+1)^2\cos(kU)\,,
\]
$U$ being distributed as Uniform $(-\pi,\pi)$. Hence the LSD of $A_n/\sqrt n$ is the free product convolution of $\mu_s$ with the law of
\[
 N+ 1+2(N+1)^{-1}\sum_{k=1}^N(N-k+1)^2\cos(kU)\,.
\]

\noindent{\bf Example 3.} By Theorem \ref{example.t2}, it follows that under the additional assumption that $\sum_{n=1}^\infty|E(G_0G_n)|<\infty$, the LSD of $A_n/\sqrt n$ is $\mu_s$.

\noindent{\bf Example 4.} Setting
\[
 \sigma:=\left(\sum_{k=-\infty}^\infty\sum_{l=-\infty}^\infty c_{k,l}^2\right)^{1/2}\,,
\]
it is easy to see from the Corollary of Theorem \ref{example.t2} that the LSD of\\ $\sigma^{-1}A_n/\sqrt n$ is $\mu_s$. Therefore, the LSD of
 $A_n/\sqrt n$ is $\tilde\mu_s$ given by
\[
 \tilde\mu_s(dx):=\frac{\sqrt{4-x^2/\sigma^2}}{2\pi\sigma}\one(|x|\le2\sigma)dx\,,
\]
which is a dilation of the WSL.

\section{An extension of Theorem \ref{example.t1}}\label{sec:ext}  In this section, we generalize Theorem \ref{example.t1} to the case when the covariances $R(u,v)$ are  not necessarily summable. Suppose that $(Z_{i,j}:i,j\in\bbz)$ is a stationary mean zero variance one Gaussian process. Define $A_n$, $\mu_n$ and $R(\cdot,\cdot)$ by \eqref{gaussian.defA_n}, \eqref{gaussian.defmu_n} and \eqref{gaussian.defR} respectively. The first assumption is, as before, that 
\begin{equation}\label{ext.assume1}
 R(u,v)=R(u,0)R(v,0),\,u,v\in\bbz\,.
\end{equation}
The second assumption, the one that replaces the summability of $R(u,v)$, is that the spectral measure of the one dimensional process $(Z_{i,0}:i\in\bbz)$ is absolutely continuous with respect to the Lebesgue measure. This means that there exists a non-negative integrable function $r$ on $[-\pi,\pi]$ such that
\begin{equation}\label{ext.ac}
 (2\pi)^{-1}\int_{-\pi}^\pi e^{inx}r(x)dx=R(n,0),\,n\in\bbz\,.
\end{equation}
Strictly speaking, $(2\pi)^{-1}r(\cdot)$ is the density of the spectral measure.
Since $R(n,0)=R(-n,0)$, it follows that $r(\cdot)$ is symmetric.
As in Theorem \ref{example.t1}, denote by $\mu_r$ the law of $r(U)$ where $U$ is an Uniform$(-\pi,\pi)$ random variable. The main result of this section is the following.

\begin{theorem}\label{ext.t1}
 As $n\to\infty$, $\mu_n$ converges weakly in probability to $\mu_r\boxtimes\mu_s$, where $\mu_s$ is the WSL.
\end{theorem}

For the proof of Theorem \ref{ext.t1}, we shall need the following lemma, which is an observation of independent interest.

\begin{lemma}\label{ext.l1}
 Define
\begin{equation}\label{ext.l1.eq1}
 c_k:=(2\pi)^{-1}\int_{-\pi}^\pi e^{ikx}\sqrt{r(x)}\,dx,\,k\in\bbz\,,
\end{equation}
the integral being defined because $\sqrt{r(\cdot)}\in L^2([-\pi,\pi])$, and real because $r(\cdot)$ is symmetric. Then,
\begin{equation}\label{ext.l1.sqsum}
 \sum_{k=-\infty}^\infty c_k^2<\infty\,,
\end{equation}
and thus the sum $\sum_{k,l\in\bbz}c_{k}c_lG_{i-k,j-l}$ converges in $L^2$ for all $i,j$, 
where $(G_{i,j}:i,j\ge1)$ is a family of i.i.d. standard Gaussian random variables. Furthermore,
\[
 (Z_{i,j}:i,j\in\bbz)\eid\left(\sum_{k,l\in\bbz}c_{k}c_lG_{i-k,j-l}:i,j\in\bbz\right)\,.
\]
\end{lemma}

\begin{proof}
  The  claim \eqref{ext.l1.sqsum} follows from an application of the Parseval identity. 
Therefore, 
\[
 \left(\sum_{k,l\in\bbz}c_{k}c_lG_{i-k,j-l}:i,j\in\bbz\right)
\]
is clearly a mean zero stationary Gaussian process. To check that it has the same finite dimensional distributions as $(Z_{i,j}:i,j\in\bbz)$, it suffices to verify that the correlations match, that is,
\begin{equation}\label{ext.l1.eq2}
 R(u,v)=\sum_{k=-\infty}^\infty\sum_{l=-\infty}^\infty c_kc_{k-u}c_lc_{l+v},\,u,v\in\bbz\,.
\end{equation}
To that end, we start with the observation that
\begin{equation}\label{ext.l1.eq3}
 \lim_{N\to\infty}\sum_{k=-N}^Nc_ke^{-ikx}=\sqrt{r(x)},\,-\pi\le x\le\pi\,,
\end{equation}
in $L^2([-\pi,\pi])$, which follows from the fact that the Fourier series of a square integrable function converges in the $L^2$ norm to that function. Therefore, the square of the left hand side converges in $L^1$ to $r(x)$ as $N\to\infty$. Consequently, for fixed $n\in\bbz$,
\begin{eqnarray*}
 \int_{-\pi}^\pi e^{inx}r(x)dx&=&\lim_{N\to\infty}\int_{-\pi}^\pi e^{inx}\left(\sum_{k=-N}^Nc_ke^{-ikx}\right)^2dx\\
&=&\lim_{N\to\infty}2\pi\sum_{k=-N}^Nc_kc_{n-k}\\
&=&2\pi\sum_{k=-\infty}^\infty c_kc_{n-k}\\
&=&2\pi\sum_{k=-\infty}^\infty c_kc_{k-n}\,,
\end{eqnarray*}
the second last equality following from \eqref{ext.l1.sqsum}, and the last equality being an outcome of the fact that $\sqrt{r(\cdot)}$ is symmetric. 
Comparing this with \eqref{ext.ac}, it follows that
\[
 R(n,0)=\sum_{k=-\infty}^\infty c_kc_{k-n},\,n\in\bbz\,.
\]
 Using \eqref{ext.assume1}, it follows that for all $u,v\in\bbz$,
\begin{eqnarray*}
 R(u,v)&=&\left(\sum_{k=-\infty}^\infty c_kc_{k-u}\right)\left(\sum_{l=-\infty}^\infty c_lc_{l-v}\right)\\
&=&\left(\sum_{k=-\infty}^\infty c_kc_{k-u}\right)\left(\sum_{l=-\infty}^\infty c_{l+v}c_l\right)\,,
\end{eqnarray*}
thereby establishing \eqref{ext.l1.eq2}. This completes the proof.
\end{proof}

\begin{proof}[Proof of Theorem \ref{ext.t1}]
Let  $(G_{i,j}:i,j\ge1)$ be a family of i.i.d. standard Gaussian random variables, and let $\{c_k\}$ be as in \eqref{ext.l1.eq1}. In view  of Lemma \ref{ext.l1}, without loss of generality, we can and do assume that
\[
 Z_{i,j}=\sum_{k,l\in\bbz}c_{k}c_lG_{i-k,j-l},\,i,j\in\bbz\,.
\]
Denote
\begin{eqnarray*}
  Z_{i,j}^{(m)}&:=&\sum_{k=-m}^m\sum_{l=-m}^mc_{k}c_lG_{i-k,j-l},\,i,j\in\bbz,\,m\ge1\,,\\
R^{(m)}(u,v)&:=&E\left[Z_{0,0}^{(m)}Z_{-u,v}^{(m)}\right],\,u,v\in\bbz,\,m\ge1\,,\\
 A_n^{(m)}&:=&((Z_{i,j}^{(m)}))_{n\times n},\,n,m\ge1\,.
\end{eqnarray*}
If $\mu_n^{(m)}$ denotes the ESD of $A_n^{(m)}/\sqrt n$, then exactly same arguments as those in the proof of Theorem \ref{inv.main} show that
\begin{equation}\label{ext.t1.eq1}
 \lim_{m\to\infty}\limsup_{n\to\infty}E\left[L^3\left(\mu_n^{(m)},\mu_n\right)\right]=0\,,
\end{equation}
where $L$ denotes the L\'evy distance. Clearly,
\begin{equation}\label{ext.t1.eq4}
\lim_{m\to\infty}R^{(m)}(0,0)=R(0,0)=1\,. 
\end{equation}
Fix $m$ large enough so that $R^{(m)}(0,0)>0$. By Theorems \ref{gaussian.t1} and \ref{example.t1}, it follows that
\begin{equation}\label{ext.t1.eq2}
 L\left(\mu_n^{(m)},\mu_{r_m}\boxtimes\mu_s\right)\prob0\,,
\end{equation}
as $n\to\infty$, where 
\[
 r_m(x):=\frac1{\sqrt{R^{(m)}(0,0)}}\sum_{n=-\infty}^\infty R^{(m)}(n,0)e^{-inx},\,-\pi\le x\le\pi\,,
\]
and $\mu_{r_m}$ denotes the law of $r_m(U)$, $U$ being an Uniform $(-\pi,\pi)$ random variable. Notice that in the sum on the right hand side, only finitely many terms are non-zero. In view of \eqref{ext.t1.eq1} and \eqref{ext.t1.eq2}, to complete the proof it suffices to show that
\begin{equation}\label{ext.t1.eq3}
 \lim_{m\to\infty}L\left(\mu_s\boxtimes\mu_{r_m},\mu_s\boxtimes\mu_r\right)=0\,.
\end{equation}

To that end, define
\[
 \tilde r_m(x):=\sum_{n=-\infty}^\infty R^{(m)}(n,0)e^{-inx}=\left(\sum_{k=-m}^m c_ke^{-ikx}\right)^2,\,-\pi\le x\le\pi\,.
\]
By \eqref{ext.l1.eq3}, it follows that
\[
 \lim_{m\to\infty}\int_{-\pi}^\pi\left|\tilde r_m(x)-r(x)\right|dx=0\,,
\]
which along with \eqref{ext.t1.eq4}
ensures that
\[
 \mu_{r_m}\weak\mu_r,\text{ as }m\to\infty\,.
\]

Define the map $\sqrt\cdot$ from the space of non-negative probability measures to that of symmetric probability measures as follows. If a non-negative probability measures $\nu$ is the law of a random variable $V$, then $\sqrt\nu$ is the law of $\epsilon\sqrt V$, where $\epsilon$ takes values $+1$ and $-1$ each with probability $1/2$ independently of $V$. Let $\cdot^2$ denote the inverse of $\sqrt\cdot$ . 
By Corollary 6.7 of \cite{bercovici:voiculescu:1993}, it follows that
\[
 \mu_s^2\boxtimes\mu_{r_m}^2\weak\mu_s^2\boxtimes\mu_r^2,\text{ as }m\to\infty\,.
\]
Lemma 8 of \cite{arizmendi:abreu:2009} tells us that for a symmetric probability measure $\nu_1$ and a non-negative probability measure $\nu_2$ such that $\nu_1(\{0\})\vee\nu_2(\{0\})<1$,
\[
 \nu_1\boxtimes\nu_2 =\sqrt{\nu_1^2\boxtimes\nu_2^2}\,.
\]
This shows \eqref{ext.t1.eq3} which, along with \eqref{ext.t1.eq1} and \eqref{ext.t1.eq2}, completes the proof.
\end{proof}

Next, let us see two examples where Theorem \ref{ext.t1} applies.

\noindent{\bf Example 5.} Let
\[
 R(u,v):=\frac{\sin u}u\frac{\sin v}v,\,u,v\in\bbz\,,
\]
where $(\sin 0)/0$ is to be interpreted as $1$. Clearly,
\[
 (2\pi)^{-1}\int_{-\pi}^\pi e^{inx}\pi\one(|x|\le1)dx=R(n,0),\,n\in\bbz\,.
\]
Therefore, by Theorem \ref{ext.t1} and the fact that the free product convolution of Bernoulli ($p$) and WSL is same as the classical product convolution of $\sqrt p$ times  Bernoulli ($p$) and WSL, it follows that the LSD in this example is the law of $\Pi W$, where $W$ follows WSL, and $\Pi$ takes the values $\sqrt\pi$ and $0$ with probabilities $1/\pi$ and $1-1/\pi$ respectively, independently of $W$ {\bf in the classical sense}.
Note that for this example, even though
\[
 \sum_{u,v}|R(u,v)|=\infty\,,
\]
the LSD is still compactly supported.

\noindent{\bf Example 6.} Define
\[
 r(x):=\frac{\sqrt\pi}2|x|^{-1/2}\one(x\neq0),\,-\pi\le x\le\pi\,,
\]
and
\[
 R(u,v):=\left((2\pi)^{-1}\int_{-\pi}^\pi e^{iux}r(x)dx\right)\left((2\pi)^{-1}\int_{-\pi}^\pi e^{ivx}r(x)dx\right),\,u,v\in\bbz\,.
\]
The fact that 
\[
 \int_{-\pi}^\pi r(x)dx=2\pi\,,
\]
ensures that the assumptions \eqref{ext.assume1} and \eqref{ext.ac} hold. Therefore, the LSD in this case is $\mu_r\boxtimes\mu_s$, where $\mu_r$ is the law of $\frac{\sqrt\pi}2|U|^{-1/2}$ and $U$ follows Uniform$(-\pi,\pi)$. The following proposition shows that the fourth moment of the LSD is infinite, which means, in particular, that the largest eigenvalue of $A_n/\sqrt n$ goes to infinity in probability.

\begin{prop}
 For any non-negative probability measure $\nu$ and integer $k\ge1$,
\begin{equation}\label{prop.eq1}
 \int_\bbr x^{2k}\mu_s\boxtimes\nu(dx)=\infty\,,
\end{equation}
if and only if
\begin{equation}\label{prop.eq2}
 \int_\bbr x^{k}\nu(dx)=\infty\,.
\end{equation}
\end{prop}

\begin{proof}
 We start with the ``if'' part, that is, assume \eqref{prop.eq2}. Let $X$ be a random variable whose law is $\nu$. For $n\ge1$, let $\nu_n$ denote the law of $X\wedge n$. In what follows, all integrals are on the whole of $\bbr$. By \eqref{eq.moments}, considering the element $\{(1,2),\ldots,(2k-1,2k)\}$ of $NC_2(2k)$, it follows that
\[
 \int x^{2k}\mu_s\boxtimes\nu_n(dx)\ge\left(\int x\nu_n(dx)\right)^k\int x^k\nu_n(dx)\,.
\]
Proposition 4.15 of \cite{bercovici:voiculescu:1993} implies that $\mu_s^2\boxtimes\nu_n^2$ is dominated by $\mu_s^2\boxtimes\nu^2$, and hence
\begin{eqnarray*}
\int x^{2k}\mu_s\boxtimes\nu(dx)&=& \int x^{k}\mu_s^2\boxtimes\nu^2(dx)\\
&\ge&\int x^{k}\mu_s^2\boxtimes\nu_n^2(dx)\\
&\ge&\left(\int x\nu_n(dx)\right)^k\int x^k\nu_n(dx)\\
&\to&\left(\int x\nu(dx)\right)^k\int x^k\nu(dx)=\infty\,,
\end{eqnarray*}
the limit in the last line following from the monotone convergence theorem. This establishes the ``if'' part.

For the ``only if'' part, assume that
\[
 \int_\bbr x^{k}\nu(dx)<\infty\,.
\]
Define $\nu_n$ as before for $n\ge1$. For $\sigma\in NC_2(2k)$, let $l_1^{\sigma},\ldots,l_{k+1}^{\sigma}$ be as in \eqref{eq.moments}. By the Skorohod embedding and Fatou's lemma, it follows that
\begin{eqnarray*}
 \int x^{2k}\mu_s\boxtimes\nu(dx)&\le&\liminf_{n\to\infty} \int x^{2k}\mu_s\boxtimes\nu_n(dx)\\
&=&\liminf_{n\to\infty}\sum_{\sigma\in NC_2(2k)}\prod_{j=1}^{k+1}\int x^{l_j^{\sigma}}\nu_n(dx)\\
&=&\sum_{\sigma\in NC_2(2k)}\prod_{j=1}^{k+1}\int x^{l_j^{\sigma}}\nu(dx)\\
&<&\infty\,,
\end{eqnarray*}
the inequality in the last line following from the observation that $l_j^{\sigma}\le k$ for all $\sigma\in NC_2(2k)$ and $1\le j\le k+1$. This completes the proof of the ``only if'' part.
\end{proof}

\section*{Acknowledgement} 
The authors are grateful to Parthanil Roy for  helpful discussions.


\begin{thebibliography}{19}
\providecommand{\natexlab}[1]{#1}
\providecommand{\url}[1]{\texttt{#1}}
\expandafter\ifx\csname urlstyle\endcsname\relax
  \providecommand{\doi}[1]{doi: #1}\else
  \providecommand{\doi}{doi: \begingroup \urlstyle{rm}\Url}\fi

\bibitem[Adamczak(2011)]{adam}
R.~Adamczak.
\newblock On the {M}archenko-{P}astur and circular laws for some classes of
  random matrices with dependent entries.
\newblock \emph{Electron. J. Probab.}, 16, 2011.

\bibitem[Anderson and Zeitouni(2008)]{AZ}
G.~W. Anderson and O.~Zeitouni.
\newblock A law of large numbers for finite-range dependent random matrices.
\newblock \emph{Comm. Pure Appl. Math.}, 61\penalty0 (8):\penalty0 1118--1154,
  2008.

\bibitem[Arizmendi and P\'erez-Abreu(2009)]{arizmendi:abreu:2009}
O.~E. Arizmendi and V.~P\'erez-Abreu.
\newblock The {S}-transform of symmetric probability measures with unbounded
  support.
\newblock \emph{Proceedings of the American Mathematical Society}, 137\penalty0
  (9):\penalty0 3057--3066, 2009.

\bibitem[Bai and Silverstein(2010)]{bai:silverstein:2010}
Z.~Bai and J.~W. Silverstein.
\newblock \emph{Spectral analysis of large dimensional random matrices}.
\newblock Springer Series in Statistics, New York, second edition, 2010.

\bibitem[Ben~Arous and Guionnet(2011)]{BA2011}
G.~Ben~Arous and A.~Guionnet.
\newblock Wigner matrices.
\newblock In \emph{The {O}xford handbook of random matrix theory}, pages
  433--451. Oxford Univ. Press, Oxford, 2011.

\bibitem[Bercovici and Voiculescu(1993)]{bercovici:voiculescu:1993}
H.~Bercovici and D.~Voiculescu.
\newblock Free convolution of measures with unbounded support.
\newblock \emph{Indiana University Mathematics Journal}, 42:\penalty0 733--773,
  1993.

\bibitem[Bose and Sen(2008)]{bose:sen:2008}
A.~Bose and A.~Sen.
\newblock Another look at the moment method for large dimensional random
  matrices.
\newblock \emph{Electronic Journal of Probability}, 13:\penalty0 588--628,
  2008.

\bibitem[Brockwell and Davis(1991)]{brockwell:davis:1991}
P.~J. Brockwell and R.~A. Davis.
\newblock \emph{Time Series: Theory and Methods}.
\newblock Springer-Verlag, New York, second edition, 1991.

\bibitem[Chatterjee(2005)]{chatterjee:2005}
S.~Chatterjee.
\newblock A simple invariance theorem.
\newblock Arxiv:math/0508213, 2005.

\bibitem[Chatterjee(2006)]{cha}
S.~Chatterjee.
\newblock A generalization of the {L}indeberg principle.
\newblock \emph{Ann. Probab.}, 34\penalty0 (6):\penalty0 2061--2076, 2006.
\newblock ISSN 0091-1798.
\newblock \doi{10.1214/009117906000000575}.
\newblock URL \url{http://dx.doi.org/10.1214/009117906000000575}.

\bibitem[G{\"o}tze and Tikhomirov(2005)]{GT}
F.~G{\"o}tze and A.~N. Tikhomirov.
\newblock Limit theorems for spectra of random matrices with martingale
  structure.
\newblock In \emph{Stein's method and applications}, volume~5 of \emph{Lect.
  Notes Ser. Inst. Math. Sci. Natl. Univ. Singap.}, pages 181--193. Singapore
  Univ. Press, Singapore, 2005.

\bibitem[Hachem et~al.(2005)Hachem, Loubaton, and Najim]{HLN}
W.~Hachem, P.~Loubaton, and J.~Najim.
\newblock The empirical eigenvalue distribution of a gram matrix: from
  independence to stationarity.
\newblock \emph{Markov Process. Related Fields}, 11\penalty0 (4):\penalty0
  629--648, 2005.
\newblock ISSN 1024-2953.

\bibitem[Hofmann-Credner and Stolz(2008)]{KM2008}
K.~Hofmann-Credner and M.~Stolz.
\newblock Wigner theorems for random matrices with dependent entries: ensembles
  associated to symmetric spaces and sample covariance matrices.
\newblock \emph{Electron. Commun. Probab.}, 13:\penalty0 401--414, 2008.

\bibitem[Nica and Speicher(2006)]{nica:speicher:2006}
A.~Nica and R.~Speicher.
\newblock \emph{Lectures on the Combinatorics of Free Probability}.
\newblock Cambridge University Press, New York, 2006.

\bibitem[Pfaffel and Schlemm(2012)]{OE2012}
O.~Pfaffel and E.~Schlemm.
\newblock Limiting spectral distribution of a new random matrix model with
  dependence across rows and columns.
\newblock \emph{Linear Algebra Appl.}, 436\penalty0 (9):\penalty0 2966--2979,
  2012.

\bibitem[Rashidi~Far et~al.(2008)Rashidi~Far, Oraby, Bryc, and Speicher]{RTWR}
R.~Rashidi~Far, T.~Oraby, W.~Bryc, and R.~Speicher.
\newblock On slow-fading {MIMO} systems with nonseparable correlation.
\newblock \emph{IEEE Trans. Inform. Theory}, 54\penalty0 (2):\penalty0
  544--553, 2008.

\bibitem[Speicher(1998)]{speiO}
R.~Speicher.
\newblock Combinatorial theory of the free product with amalgamation and
  operator-valued free probability theory.
\newblock \emph{Mem. Amer. Math. Soc.}, 132\penalty0 (627):\penalty0 x+88,
  1998.
\newblock ISSN 0065-9266.

\bibitem[Speicher(2011)]{spei-review}
R.~Speicher.
\newblock Free probability theory.
\newblock In \emph{The {O}xford handbook of random matrix theory}, pages
  452--470. Oxford Univ. Press, Oxford, 2011.

\bibitem[Wigner(1958)]{wig58}
E.~P. Wigner.
\newblock On the distribution of the roots of certain symmetric matrices.
\newblock \emph{Ann. of Math. (2)}, 67:\penalty0 325--327, 1958.

\end{thebibliography}

\end{document}